\documentclass[preprint, 12pt, sort&compress]{elsarticle}

\usepackage{latexsym}
\usepackage{amsfonts}
\usepackage{graphicx,amssymb,natbib}
\usepackage{amsmath}
\usepackage{amssymb}
\usepackage[mathscr]{eucal}
\usepackage{enumerate}
\usepackage{amsthm}

\usepackage{bm}

\usepackage{color, soul} %ÓÃcolor, ºÍ soul °ü

\begin{document}

\newtheorem{tm}{Theorem}[section]
\newtheorem{pp}{Proposition}[section]
\newtheorem{lm}{Lemma}[section]
\newtheorem{df}{Definition}[section]
\newtheorem{tl}{Corollary}[section]
\newtheorem{re}{Remark}[section]
\newtheorem{eap}{Example}[section]

\newcommand{\pof}{\noindent {\bf Proof} }
\newcommand{\ep}{$\quad \Box$}

\newcommand{\al}{\alpha}
\newcommand{\be}{\beta}
\newcommand{\var}{\varepsilon}
\newcommand{\la}{\lambda}
\newcommand{\de}{\delta}
\newcommand{\str}{\stackrel}

\renewcommand{\proofname}{\bf Proof}
\setcounter{section}{-1}

\allowdisplaybreaks

\begin{frontmatter}

\title{Characterizations of
 endograph metric and $\Gamma$-convergence on fuzzy sets
\tnoteref{t1}
 }
\tnotetext[t1]{Project supported by
 National Natural Science
Foundation of China (No. 61103052), and by Natural Science Foundation of Fujian Province of China(No. 2016J01022)}
\author{Huan Huang}
 \ead{hhuangjy@126.com }
\address{Department of Mathematics, Jimei
University, Xiamen 361021, China}

%\author[rvt]{C.V.~Radhakrishnan\corref{cor1}\fnref{fn1}}
%\ead{cvr@river-valley.com}
%\address[rvt]{River Valley Technologies, SJP Building,
%Cotton Hills, Trivandrum, Kerala, India 695014}

\date{}
%\maketitle

\begin{abstract}

 This
 paper is devoted to the relationships and properties of the endograph metric and the $\Gamma$-convergence.
The main contents can be divided into three closely related parts.
Firstly, on the class of upper semi-continuous fuzzy sets with bounded $\alpha$-cuts, we find
that an endograph metric convergent sequence is exactly a
$\Gamma$-convergent sequence satisfying the condition that
the union of $\alpha$-cuts of all its elements is a bounded set in $\mathbb{R}^m$ for each $\alpha > 0$.
Secondly, based on investigations of level characterizations of fuzzy sets themselves,
 we present
level characterizations (level decomposition properties) of
the endograph metric and the $\Gamma$-convergence.
 It is worth mentioning that, using the condition and the level characterizations given above, we discover the fact:
the endograph metric and the $\Gamma$-convergence
are compatible
on
a large class of general fuzzy sets which do not have any assumptions of normality, convexity or star-shapedness.
Its subsets
 include
 common particular fuzzy
sets such as fuzzy numbers (compact and noncompact), fuzzy star-shaped
numbers (compact and noncompact), and general fuzzy star-shaped numbers (compact and noncompact).
Thirdly,
 on
the
basis
of the
conclusions presented above,
we study various
subspaces of the space of upper semi-continuous fuzzy sets with bounded $\alpha$-cuts
equipped with the endograph metric. We present characterizations
  of
   total boundedness, relative compactness and compactness in these fuzzy set spaces
and
 clarify
 relationships
 among these fuzzy set spaces.
It
is
pointed out
that
the fuzzy set spaces of noncompact type   are exactly the completions of their compact counterparts under the endograph metric.

\end{abstract}


\end{frontmatter}

\section{Basic notions} \label{bns}

Let $\mathbb{N}$ be the set of all natural numbers,
$\mathbb{Q}$
 be the set of all rational numbers,
$\mathbb{R}^m$
be the $m$-dimensional Euclidean space,
$K(  \mathbb{R}^m  )$
be
the
set of all nonempty compact sets in $\mathbb{R}^m$,
$K_C(  \mathbb{R}^m  )$
be
the
set of all nonempty compact convex sets in $\mathbb{R}^m$
and
$C(  \mathbb{R}^m  )$
be
the
set of all nonempty closed sets in $\mathbb{R}^m$.

A set $K\in K(\mathbb{R}^m)$ is said to be star-shaped relative to a point $x\in K$
if
for each $y\in K$, the line $\overline{xy}$ joining $x$ to $y$
is contained in $K$.
The kernel ker\,$K$ of $K$ is the set of all points $x\in K$
such that
$\overline{xy}\subset K$
for each $y\in K$.
The symbol $K_S(\mathbb{R}^m)$
is used to
denote
all the star-shaped sets in $\mathbb{R}^m$.

Obviously,
$K_C(\mathbb{R}^m)\subsetneq K_S(\mathbb{R}^m)$.
It can be checked that
$\hbox{ker}\; K \in K_C(\mathbb{R}^m) $ for all $K\in K_S(\mathbb{R}^m)$.

A fuzzy set $u$ on   $\mathbb{R}^m$ is in fact a function $u$ from $\mathbb{R}^m$ to [0,1].
The symbol
  $F(\mathbb{R}^m)$
is used to represent all
fuzzy sets on $\mathbb{R}^m$
(see \cite{wu, da} for details).
$
\mathrm{2}^{\mathbb{R}^m} := \{  S: S\subseteq \mathbb{R}^m      \}
$
can be embedded in $F (\mathbb{R}^m)$, as any $S \subset \mathbb{R}^m$ can be
seen as its characterization function, i.e. the fuzzy set
\[
\widehat{S}(x)=\left\{
\begin{array}{ll}
1,x\in S, \\
0,x\notin S.
\end{array}
\right.
\]
Specially, $\widehat{\emptyset}_m$ represents the fuzzy set on $\mathbb{R}^m$ which is defined by
$
\widehat{\emptyset}_m (x) \equiv 0
$
for any     $x\in \mathbb{R}^m$.
For simplicity,
we
denote $\widehat{\emptyset}_m $ by $\widehat{\emptyset}$
if there is no confusion.

For
$u\in F(\mathbb{R}^m)$, let $\{u>\al\}$
denote
the strong $\al$-cut of $u$, i.e.
$\{u>\al\} = \{x\in \mathbb{R}^m: u(x) > \al \}$, and let $[u]_{\al}$ denote the $\al$-cut of
$u$, i.e.
\[
[u]_{\al}=\begin{cases}
\{x\in \mathbb{R}^m : u(x)\geq \al \}, & \ \al\in(0,1],
\\
{\rm supp}\, u=\overline{    \{ u > 0 \}    }, & \ \al=0.
\end{cases}
\]

For
$u\in F(\mathbb{R}^m)$,
we suppose that
\\
(\romannumeral1) \ $u$ is upper semi-continuous;
\\
(\romannumeral2) \ $u$ is normal: there exists at least one $x_{0}\in \mathbb{R}^m$
with $u(x_{0})=1$;
\\
(\romannumeral3-1) \ $u$ is fuzzy convex: $u(\la x+(1-\la)y)\geq {\rm min} \{u(x),u(y)\}$
for $x,y \in \mathbb{R}^m$ and $\la \in [0,1]$;
\\
(\romannumeral3-2) \ $u$ is fuzzy star-shaped, i.e.,   there exists $x\in \mathbb{R}^m$
such that
$u$ is fuzzy star-shaped with respect to $x$, namely,
$u(\lambda y + (1-\lambda) x)  \geq  u(y)$
for all $y\in \mathbb{R}^m$ and $\la \in [0,1]$;
\\
(\romannumeral3-3) \ If $\lambda\in (0,1]$, then there exists
$x_\lambda \in [u]_\lambda$
such that
$\overline{x_\lambda y} \in [u]_\lambda$
for all $y\in [u]_\lambda$;
\\
(\romannumeral3-4) \  $[u]_\alpha$ is a connected set in $\mathbb{R}^m$ for each $\al\in (0,1]$;
\\
(\romannumeral4-1) \ $[u]_0$ is a bounded set in $\mathbb{R}^m$;
\\
(\romannumeral4-2) \  $[u]_\alpha$ is a bounded set in $\mathbb{R}^m$ for each $\al\in (0,1]$.
\begin{itemize}
 \item  We use the symbol $F_{USC} (\mathbb{R}^m)$
   to
   denote the set of all fuzzy sets $u$ on $\mathbb{R}^m$ with $u$ satisfying (\romannumeral1).

  \item  We use the symbol $F_{USCG} (\mathbb{R}^m)$
   to
   denote the set of all fuzzy sets $u$ on $\mathbb{R}^m$ with $u$ satisfying
   (\romannumeral1) and (\romannumeral4-2).

  \item  We use the symbol $F_{USCB} (\mathbb{R}^m)$
   to
   denote the set of all fuzzy sets $u$ on $\mathbb{R}^m$ with $u$ satisfying
   (\romannumeral1) and (\romannumeral4-1).

\item We use the symbol $F_{USCGCON} (\mathbb{R}^m)$
   to
   denote the set of all fuzzy sets $u$ on $\mathbb{R}^m$ with $u$ satisfying
    (\romannumeral1), (\romannumeral3-4) and (\romannumeral4-2).
    So
   $u \in F_{USCGCON} (\mathbb{R}^m)$
   if and only if
   $u \in F_{USCG} (\mathbb{R}^m)$ and each cut-set of $u$ is a connected set in $\mathbb{R}^m$.

\end{itemize}
It's easy to see
that
$
F_{USCGCON} (\mathbb{R}^m),   F_{USCB}   \subsetneq   F_{USCG} (\mathbb{R}^m)   \subsetneq F_{USC} (\mathbb{R}^m) $.

 $F_{USC} (\mathbb{R}^m)$,  $F_{USCG} (\mathbb{R}^m)$ and    $F_{USCB} (\mathbb{R}^m)$  are said to be      \textbf{general} fuzzy sets.
The other types of
fuzzy sets mentioned in this paper
are
called
\textbf{particular} fuzzy sets.
The $\al$-cuts of general fuzzy sets have \textbf{no} assumptions of
normality,    convexity, starshapedness or even connectedness.

Here
we list
several common subsets \cite{da,du,wu,qiu} of $F_{USCGCON} (\mathbb{R}^m)$.
\begin{itemize}
  \item
  If
     $u$ satisfies (\romannumeral1), (\romannumeral2), (\romannumeral3-1) and (\romannumeral4-1),
  then
   $u$ is a (compact) fuzzy number. The set of all fuzzy numbers is denoted by $E^m$.

  \item
  If $u$ satisfies (\romannumeral1), (\romannumeral2), (\romannumeral3-2) and (\romannumeral4-1),
  then
   $u$ is a (compact) fuzzy  star-shaped number. The set of all fuzzy star-shaped  numbers is denoted by
  $S^m$.

\item
If $u$ satisfies (\romannumeral1), (\romannumeral2), (\romannumeral3-3) and (\romannumeral4-1),
  then
   $u$ is a (compact) general fuzzy star-shaped number. The set of all  general fuzzy  star-shaped  numbers is denoted by
  $\widetilde{S}^m$.

\end{itemize}
Fuzzy
 number has been exhaustively studied in both theory and applications \cite{wang,wang2,wa,da,wu,huang11}.
$\mathbb R^m$ can be embedded in $E^m$, as any $r \in  \mathbb{R}^m$ can be
viewed as the fuzzy number
\[
r(x)=\left\{
\begin{array}{ll}
1,x=r, \\
0,x\not=r.
\end{array}
\right.
\]

Noncompact counterparts of the three types of compact fuzzy sets mentioned above
are
listed below.
\begin{itemize}
  \item   If
     $u$ satisfies (\romannumeral1), (\romannumeral2), (\romannumeral3-1) and (\romannumeral4-2),
  then
   $u$ is a noncompact fuzzy number. The set of all noncompact fuzzy numbers is denoted by $E^m_{nc}$.

 \item
  If $u$ satisfies (\romannumeral1), (\romannumeral2), (\romannumeral3-2) and (\romannumeral4-2),
  then
   $u$ is a noncompact fuzzy star-shaped number. The set of all noncompact fuzzy star-shaped  numbers is denoted by
  $S^m_{nc}$.

  \item
If $u$ satisfies (\romannumeral1), (\romannumeral2), (\romannumeral3-3) and (\romannumeral4-2),
  then
   $u$ is a noncompact general fuzzy star-shaped number. The set of all noncompact
     general fuzzy  star-shaped  numbers is denoted by
  $\widetilde{S}^m_{nc}$.
\end{itemize}
Clearly,
$
E^m \subsetneq    E^{m}_{nc}
$,
$
  S^m \subsetneq    S^{m}_{nc}
$,
$\widetilde{S}^{m} \subsetneq    \widetilde{S}^{m}_{nc}$,
$ E^{m} \subsetneq  S^{m}   \subsetneq   \widetilde{S}^{m} $
and
$ E^{m}_{nc} \subsetneq  S^{m}_{nc}   \subsetneq   \widetilde{S}^{m}_{nc} $.
For
simplicity, we call
these compact and noncompact common particular fuzzy sets
\textbf{common fuzzy sets} in the sequel.
That is to say, in this paper, common fuzzy sets are refer to
$
E^m $, $  E^{m}_{nc}$,
$
  S^m $, $   S^{m}_{nc}$,
$\widetilde{S}^{m} $, and $   \widetilde{S}^{m}_{nc}$.

A
kind of noncompact fuzzy set may be obtained using
the
weaker assumption (iv-2) instead of (iv-1) on the corresponding kind of compact fuzzy set. Therefore the latter is a subset of the former.
We can see that $F_{USCG} (\mathbb{R}^m)$ is the noncompact counterpart of $F_{USCB} (\mathbb{R}^m)$.

Most of the metrics or topologies on fuzzy set
are
based on the Hausdorff metric or the Kuratowski convergence.

The
well-known
\emph{\textbf{Hausdorff metric}} \bm{$H$} on
   $C(\mathbb{R}^m)$ is defined by:
$$H(U,V)=\max\{H^{*}(U,V),\ H^{*}(V,U)\}$$
for arbitrary $U,V\in K(\mathbb{R}^m)$,
where
  $$
H^{*}(U,V)=\sup\limits_{u\in U}\,d\, (u,V) =\sup\limits_{u\in U}\inf\limits_{v\in
V}d\, (u,v).
$$
The Hausdorff metric can be extended to a metric on $C(\mathbb{R}^m) \cup \{\emptyset\}$
with
\[  H(M_1, M_2) =\left\{
                \begin{array}{ll}
                 H(M_1, M_2) , & \hbox{if} \  M_1, M_2 \in C(\mathbb{R}^m),
                   \\
                  +\infty, & \hbox{if} \  M_1=\emptyset \   \hbox{and} \ M_2 \in C(\mathbb{R}^m),
                   \\
                  0, & \hbox{if}  \  M_1=M_2=\emptyset.
                \end{array}
              \right.
 \]

Let
$(X,d)$ be a metric space.
We say that a sequence of sets $\{C_n\}$ \emph{\textbf{Kuratowski converges}} to
$C\subseteq X$, if
$$
C
=
\liminf_{n\rightarrow \infty} C_{n}
=
\limsup_{n\rightarrow \infty} C_{n},
$$
where
\begin{gather*}
\liminf_{n\rightarrow \infty} C_{n}
 =
 \{x\in X: \  x=\lim\limits_{n\rightarrow \infty}x_{n},    x_{n}\in C_{n}\},
\\
\limsup_{n\rightarrow \infty} C_{n}
=
\{
 x\in X : \
 x=\lim\limits_{j\rightarrow \infty}x_{n_{j}},x_{n_{j}}\in C_{n_j}
\}
 =
 \bigcap\limits_{n=1}^{\infty}   \overline{   \bigcup\limits_{m\geq n}C_{m}    }.
\end{gather*}
In this case, we'll write
\bm{  $C=\lim_{n\to\infty}C_n  (Kuratowski)$    } or \bm{  $C=\lim_{n\to\infty}C_n  (K)$    } for simplicity.

The \textbf{\emph{Fell topology }}\bm{$\tau_f$} was introduced by Fell \cite{fell}.
It
is
known
that
the Fell topology $\tau_f$
is
 compatible with the Kuratowski convergence
on
 $C(\mathbb{R}^m) \cup \{\emptyset\}$. In general,
the Fell topology convergence is weaker than
the Hausdorff metric convergence.

Given a fuzzy set $u$ on $\mathbb{R}^m$,
define
$${\rm end}\, u:= \{ (x, t)\in \mathbb{R}^m\times [0,1]: u(x) \geq t\}.$$
${\rm end}\, u$
is called the \textbf{\emph{endograph}} of $u$, which is also called the \emph{hypograph} of $u$ in many references.
A fuzzy set
can
be
identified with its endograph.
Given $u$ in $F (\mathbb{R}^m)$,
it is known
that
 $u$ is in
$F_{USC} (\mathbb{R}^m)$ if and only if
${\rm end}\, u$
is
 in $C(\mathbb{R}^m \times [0,1])$ (see \cite{kloeden} for details).

Kloeden \cite{kloeden2} introduced the \emph{\textbf{endograph metric}} \bm{$H_{\rm end}$} on upper semi-continuous fuzzy sets,
which
is defined by
$$
H_{\rm end}(u,v): =  H({\rm end}\, u,  {\rm end}\, v )
$$
for all $u,v \in F_{USC} (\mathbb{R}^m)$.

Rojas-Medar and Rom\'{a}n-Flores \cite{rojas} introduced
the
$\Gamma$- convergence
on
upper semi-continuous fuzzy sets:

  Let $u$, $u_n$, $n=1,2,\ldots$, be fuzzy sets in $F_{USC} (\mathbb{R}^m)$.
  Then $u_n$ \bm{$\Gamma$}\textbf{-converges}
  to
  $u$ (\bm{$u_n \str{\Gamma}{\longrightarrow} u$})
  if
  $${\rm end}\, u= \lim_{n\to \infty} {\rm end}\, u_n (K).$$

We can see
that the
endograph metric and the $\Gamma$-convergence on $F_{USC} (\mathbb{R}^m)$ are
in fact
Hausdorff metric and Fell topology convergence (Kuratowski convergence)
equipped
on the set of endographs of fuzzy sets in $F_{USC} (\mathbb{R}^m)$, respectively.

Diamond and Kloeden \cite{da} introduced
an $L_p$-type metric:
 the \bm{$d_p$} \textbf{metric} ($1\leq p<\infty$)
which is defined by
\begin{equation*}\label{dpm}
d_p\, (u,v)=\left(     \int_0^1 H([u]_\al, [v]_\al) ^p  \; d\alpha  \right)^{1/p}
\end{equation*}
for all $u,v\in F_{USC} (\mathbb{R}^m)$.

\section{Introduction}

Kloeden
\cite{kloeden2}
introduced the endograph metric $H_{\rm end}$    on
fuzzy sets.
Rojas-Medar and Rom\'{a}n-Flores \cite{rojas} introduced
the
$\Gamma$-convergence
on
fuzzy sets.
These two
convergence structures are based on Hausdorff metric and Fell topology convergence (Kuratowski convergence), respectively.
 The
  Hausdorff metric and the Fell topology have been exhaustively studied \cite{wan, ngu, yang5, matheron}.
In recent years, studies
involve
the endograph metric $H_{\rm end}$ and the $\Gamma$-convergence have received considerable attention
 from points of view of theory and applications
 \cite{fan2, kloeden, kloeden2, huang7, yang, greco2, kupka, can, pedraza, rojas}.
However,
research on these two convergence structures
is
not much.
In this paper, we find that these two kinds of convergence structures on $F_{USC} (\mathbb{R}^m)$
have
 deep links and discuss their properties.
This paper
has
 three closely related themes.

The work of
Wang and Wei \cite{wan} indicated
 that
 the
$\Gamma$-convergence
on $F_{USC}(\mathbb{R}^m)$
 is metrizable.
On the other hand,
it is easy
to
check
that
the $\Gamma$-convergence is weaker than
the
 endograph metric convergence
on
$F_{USC}(\mathbb{R}^m)$.
However,
Huang and Wu \cite{huang7}
found the interesting fact
that
the $\Gamma$-convergence is equivalent to the endograph metric convergence on $E^1$, the set of one-dimensional compact fuzzy numbers.
Since
the endograph metric is intuitive and is increasingly used in fuzzy systems \cite{kloeden,kloeden2,can},
it is
worthwhile
to
consider the following questions:
the relation of
  these two types of convergence; on what kinds of fuzzy sets these two types of convergence are equivalent; whether these two types of convergence can be equivalent on common fuzzy sets.
This is the first theme of this paper.

The study
of
level characterizations of convergence structures on fuzzy sets
is basic and important
and
can help us to see the original convergence structures more deeply
and
clearly.
 Rojas-Medar and Rom\'{a}n-Flores
have pointed out a
level characterization
of
$\Gamma$-convergence on fuzzy numbers (Proposition 3.5 in \cite{rojas}).
Based on this,
Huang and Wu \cite{huang7}
have realized and proven
 that
  the $\Gamma$-convergence (endograph metric convergence, resp.) can be decomposed to the Fell topology convergence (Hausdorff metric convergence, resp.) on some $\al$-cuts.
  At
 the same time and independently, Fan \cite{fan3} also obtained the same level characterizations (level decomposition properties) of the endograph metric.
These
level characterizations
were used
to
discuss the relationships among
level convergence, $d_p$ convergence and $\Gamma$-convergence
 \cite{huang7, fan3}.
 The level characterizations in \cite{huang7,fan3}
 are given in
 the setting of
 one-dimensional noncompact fuzzy numbers.
  Trutschnig \cite{trutschnig}
 have
 proven an important fact that the platform points of $u$ is countable when $u\in E^m_{nc}$.
Use this fact and proceed according to \cite{huang7} or \cite{fan3},
it follows
that
 these level characterizations are still true in the settings of $m$-dimensional fuzzy numbers.
  However,
in theory and applications \cite{kloeden, chanussot, ba, du}, fuzzy set classes much larger than fuzzy numbers are often used,
such as
 fuzzy star-shaped numbers, upper semi-continuous fuzzy sets, etc.
Hence it is
natural and important to consider
the problem
 whether
these level characterizations still hold on larger fuzzy set classes,
for instance,
 general fuzzy sets whose $\al$-cuts may not have assumptions of
normality,    convexity, starshapedness or even connectedness.
The study of this problem will provide a more clear understanding
of
the level characterizations in the setting of fuzzy numbers given in previous works.
This
problem
involves
 level structural characteristics of fuzzy sets themselves
including the number of platform points of a general fuzzy set,
the Fell topology continuity and the endograph metric continuity of
the cut-set functions of
general fuzzy sets.
This is the second theme of this paper.

Compactness is one of the central concepts in topology and analysis (see \cite{kelley}). Characterizations of compactness
are useful
in theoretical research and practical applications \cite{bu,gh3,huang3,huang6,huang8,kloeden,wa,zeng}.
Many researches are devoted to characterizations
of
compactness in a variety of fuzzy set spaces endowed with different topologies \cite{da,fan,greco,greco3,huang,huang7,huang9,ma,roman,trutschnig,wu2,zhao}.
In recent years, the endograph metric
has
 attracted more and more attention. For instance,
 Kloeden
and Lorenz \cite{kloeden} have   established
a Peano theorem for
fuzzy differential equations based on the use of endograph metric.
 Kupka \cite{kupka}
gave
an approximation procedure
for
Zadeh's extension
of a
continuous map
 in the sense of endograph metric.
These works
indicate that the endograph metric has
 significant advantages in many situations.
So
 it
is
needed in both theory and applications to present
characterizations
of
compactness for various kinds of fuzzy set spaces endowed with the endograph metric
and consider
relations of these fuzzy set spaces.
This is the third theme
of
this paper.

These three themes are, in fact, interrelation. The investigations of these themes
 are
carried out alternately and the relationships of these themes are summarized in the last section.

 The organization of this paper is as follows.
 In
 Sections 2 and 3,
 we
introduce and discuss
Hausdorff metric, Fell topology, Kuratowski convergence and fuzzy sets.
In
Section 4,
 we
  investigate the relationship between the endograph metric convergence and the $\Gamma$-convergence on $F_{USCG}(\mathbb{R}^m)$.
In
Section 5,
we study level structural characteristics of fuzzy sets themselves.
Then, in Section 6,
we
present
level characterizations of the $\Gamma$-convergence and that of the endograph metric convergence.
As an application, we discuss the relationships among $d_p$ metric, endograph metric and $\Gamma$-convergence on $F_{USC} (\mathbb{R}^m)$.
Based
 on
  results given above, in Sections 7 and 8,
we
give characterizations of relative
compactness, total boundedness, and compactness
in
$(F_{USCG} (\mathbb{R}^m), H_{\rm end})$
and
$(F_{USCB} (\mathbb{R}^m), H_{\rm end})$.
It is shown
that
the former is the completion of the latter.
In Section 9,
we
reconsider
the
level characterizations
given in Section 6 when
 restricted
 to
  $F_{USCGCON} (\mathbb{R}^m) \backslash \widehat{\emptyset}$ and when restricted to common fuzzy sets mentioned here, respectively.
It is
found
that
the
endograph metric convergence
 and the $\Gamma$-convergence
are equivalent on
 $F_{USCGCON} (\mathbb{R}^m) \backslash \widehat{\emptyset}$, which includes any common fuzzy set mentioned here.
By using these conclusions, we revisit the results in previous works.
Then,
we
clarify the relationships among
various
subspaces
of
$(F_{USCG} (\mathbb{R}^m), H_{\rm end})$.
Finally,
we
 give
characterizations of relative compactness, total boundedness, and compactness
in
these subspaces.
At last, we draw our conclusions in Section 10.

\section{Hausdorff metric, Fell topology and Kuratowski convergence \label{hfkc}}

Hausdorff metric, Fell topology and Kuratowski convergence are used widely and have been exhaustively studied.
All
of
these three structures can be equipped on  $C(\mathbb{R}^m) \cup \{\emptyset\}$.

The Fell topology $\tau_f$ was introduced by Fell \cite{fell}. It is
also known as H-topology, hit-or-miss topology, Choquet-Matheron
topology, or weak Vietoris topology \cite{wan}.

The Fell topology   $\tau_f$
 on
 $C(\mathbb{R}^m) \cup \{\emptyset\}$ is metrizable.
The Fell topology $\tau_f$ is compatible with the Kuratowski convergence
on
 $C(\mathbb{R}^m) \cup \{\emptyset\}$. In general,
the Fell topology convergence is weaker than
the Hausdorff metric convergence.
The
readers can see
\cite{trutschnig, wan, ngu, matheron}
or references therein
for details.

Some known results about the Hausdorff metric are listed below.

\begin{pp} \cite{da} \label{kcs}
 $(C(\mathbb{R}^m), H)$ is a complete metric space, in which $K(\mathbb{R}^m)$ and $K_C(  \mathbb{R}^m  )$
are closed subsets.
Hence $(K(\mathbb{R}^m), H)$ and $(K_C(\mathbb{R}^m), H)$ are also complete metric spaces.
\end{pp}

\begin{pp}
 \cite{da, roman} \label{sca}
A
nonempty subset $U$ of $(   K(\mathbb{R}^m),    H     )$
is
compact
if and only if
it is closed and bounded in $(   K(\mathbb{R}^m),    H     )$.
 \end{pp}

\begin{pp}\cite{da} \label{mce}
  Let $\{u_n\} \subset K(\mathbb{R}^m)$ satisfy
$ u_1 \supseteq u_2 \supseteq \ldots \supseteq u_n \supseteq \ldots.  $
Then
$u= \bigcap_{n=1}^{+\infty}  u_n \in K(\mathbb{R}^m)$ and
$H(u_n, u) \to 0  \ \hbox{as} \ n\to \infty$.

On the other hand, if $ u_1 \subseteq u_2 \subseteq \ldots \subseteq u_n \subseteq \ldots  $
and
$u= \overline{\bigcup_{n=1}^{+\infty}  u_n }\in K(\mathbb{R}^m)$,
then
$H(u_n, u) \to 0  \ \hbox{as} \ n\to \infty$.
\end{pp}

The following two known propositions reveal the relation of the Hausdorff metric convergence
and
the
 Kuratowski convergence. The readers can see \cite{greco2} for details\footnote{We are not able to obtain reference \cite{greco2}.  Propositions \ref{hms} and \ref{klhe} were introduced as known statements in another paper, and this paper claimed that these two propositions come from \cite{greco2}. There are some other references that we can not obtain, such as \cite{fell, pr}.}.

\begin{pp} \label{hms}
  Suppose that $C$, $C_n$, $n=1,2,\ldots$, are nonempty compact sets
  in $\mathbb{R}^m$.
  Then
  $H(C_n, C) \to 0$ as $n\to \infty$
  implies that
    $C=\lim_{n\to\infty}C_n  (K)$.
\end{pp}

\begin{pp}  \label{klhe}
  Suppose that $C$, $C_n$, $n=1,2,\ldots$, are nonempty compact sets
  in $\mathbb{R}^m$
  and
  that $C_n$, $n=1,2,\ldots$, are connected sets.
  If
    $C=\lim_{n\to\infty}C_n (K)$,
  then
   $H(C_n, C) \to 0$ as $n\to \infty$.
\end{pp}

We need the following conclusions, which will be used in the sequel of this paper.

\begin{lm} \label{infe}
Let $(X,d)$ be a metric space,
and
$C_{n}$, $n=1,2,\ldots$, be a sequence of sets in $X$.
Suppose that $x\in X$.
Then
\\
(\romannumeral1) \
  $x \in \liminf_{n\rightarrow \infty} C_{n}$
  if and only if
$\lim_{n\to \infty} d(x, C_n) = 0$,
\\
(\romannumeral2) \
  $x \in \limsup_{n\rightarrow \infty} C_{n}$
  if and only if there is a subsequence $\{C_{n_k}\}$ of $\{C_n\}$
such
that
$\lim_{k\to \infty} d(x, C_{n_k}) = 0$.
\end{lm}

\begin{proof}
  \ The desired results follow from the definitions of $\liminf_{n\rightarrow \infty} C_{n}$ and $\limsup_{n\rightarrow \infty} C_{n}$.
\end{proof}

\begin{tm} \label{infc}
Let $(X,d)$ be a metric space,
and
$C_{n}$, $n=1,2,\ldots$, be a sequence of sets in $X$.
  Then
    $\liminf_{n\rightarrow \infty} C_{n}$ and $\limsup_{n\rightarrow \infty} C_{n}$ are closed sets.
\end{tm}

\begin{proof}
 \  We only show $\liminf_{n\rightarrow \infty} C_{n}$ is a closed set.
The closedness of $\limsup_{n\rightarrow \infty} C_{n}$ can be proved in a similar manner.

Denote $C:  = \liminf_{n\rightarrow \infty} C_{n}$.
If $C=\emptyset$,
then, obviously, $C$ is a closed set.
If
$C \not=  \emptyset$.
Suppose that $x\in \overline{C}$, then there exists $x_n \in C$, $n=1,2,\ldots$, such that
$x=\lim_{n\to \infty} x_n$. So for each $\varepsilon>0$,
we can find a $x_k$ such
that
$$d(x, x_k) < \varepsilon/2.$$
Since $x_k\in C= \liminf_{n\rightarrow \infty} C_{n}$,
by Lemma \ref{infe},
there is an
$N(x_k, \varepsilon) \in \mathbb{N}$
satisfying
that
 $$d(x_k, C_n) < \varepsilon/2$$
  for all $n\geq N$.
 Thus
$$d(x, C_n) < \varepsilon$$
  for all $n\geq N$.
From the arbitrariness of $\varepsilon>0$ and Lemma \ref{infe},
we know
$x\in C$.
\end{proof}

\begin{tm} \label{ksc}
$K_S(\mathbb{R}^m)$ is a closed set in $( K(\mathbb{R}^m),   H )$.
\end{tm}

\begin{proof}
 This is Theorem 2.1 in \cite{huang9}.
\end{proof}

\begin{tl} \label{kef}
  Let $C$, $C_n$ be star-shaped sets, $n=1,2,\ldots$.
 If $H(C_n, C) \to 0$,
then
$\limsup_{n\to \infty}\mbox{ker}\; C_n \subset \mbox{ker}\; C$.
\end{tl}

\begin{proof}
  \ This is Corollary 2.1 in \cite{huang9}.
\end{proof}

\begin{re}
{\rm
  We do not know whether Theorems \ref{infc}, \ref{ksc} and Corollary \ref{kef} are known conclusions,
so
we give our proofs at here or in \cite{huang9}.
}
\end{re}

\section{Fuzzy set spaces \label{fuzset}}

$F_{USC}(\mathbb{R}^m)$
is
the most often used kind of fuzzy sets in theory and applications.
For different practical needs,
people
present
various types of common fuzzy sets on $\mathbb{R}^m$ by attaching some additional assumptions of normality, fuzzy convexity, fuzzy star-shapedness and so on.
 Common fuzzy sets
include
 fuzzy numbers (compact and noncompact),
  fuzzy   star-shaped   numbers (compact and noncompact),
and
general fuzzy star-shaped numbers (compact and noncompact).
All of the common fuzzy sets mentioned above
are
subsets
 of
  $F_{USCGCON} (\mathbb{R}^m)$.

The following representation theorem is used
widely in the theory and applications of fuzzy numbers.

\begin{pp}\cite{nr} \label{nr}\
Given $u\in E^m$ ($u\in E^m_{nc}$). Then
\\
(\romannumeral1) \ $[u]_\la\in K_C(\mathbb{R}^m)$ for all $\la\in [0,1]$ ($\la\in (0,1]$);
\\
(\romannumeral2) \ $[u]_\la=\bigcap_{\gamma<\lambda}[u]_\gamma$ for all $\la\in (0,1]$;
\\
(\romannumeral3) \ $[u]_0=\overline{\bigcup_{\gamma>0}[u]_\gamma}$.

Moreover, if the family of sets $\{v_\al:\al\in [0,1]\}$ satisfies
conditions $(\romannumeral1)$ through $(\romannumeral3)$, then there exists a unique $u\in E^m$ ($u\in E^m_{nc}$)
such that $[u]_{\la}=v_\lambda$ for each $\la\in [0,1].$
\end{pp}

Similarly, we can obtain representation theorems for   $S^m$,
$S^m_{nc}$,
$\widetilde{S}^m$,
$\widetilde{S}^{m}_{nc}$, $F_{USCB} (\mathbb{R}^m)$ and $F_{USCG} (\mathbb{R}^m)$.

\begin{tm} \label{smre}\
Given $u\in S^m$ ($u\in S^m_{nc}$). Then
\\
(\romannumeral1) \ $[u]_\la\in K_S(\mathbb{R}^m)$ for all $\la\in [0,1]$ ($\la\in (0,1]$),  and $\bigcap_{\lambda \in (0,1]} \mbox{ker}\, [u]_\lambda \not= \emptyset$;
\\
(\romannumeral2) \ $[u]_\la=\bigcap_{\gamma<\lambda}[u]_\gamma$ for all $\la\in (0,1]$;
\\
(\romannumeral3) \ $[u]_0=\overline{\bigcup_{\gamma>0}[u]_\gamma}$.

Moreover, if the family of sets $\{v_\al:\al\in [0,1]\}$ satisfies
conditions $(\romannumeral1)$ through $(\romannumeral3)$, then there exists a unique $u\in S^m$ ($u\in S^m_{nc}$)
such that $[u]_{\la}=v_\lambda$ for each $\la\in [0,1].$
\end{tm}

\begin{tm} \label{gsmre}\
Given $u\in \widetilde{S}^m$ ($u\in \widetilde{S}^m_{nc}$). Then
\\
(\romannumeral1) \ $[u]_\la\in K_S(\mathbb{R}^m)$ for all $\la\in [0,1]$ ($\la\in (0,1]$);
\\
(\romannumeral2) \ $[u]_\la=\bigcap_{\gamma<\lambda}[u]_\gamma$ for all $\la\in (0,1]$;
\\
(\romannumeral3) \ $[u]_0=\overline{\bigcup_{\gamma>0}[u]_\gamma}$.

Moreover, if the family of sets $\{v_\al:\al\in [0,1]\}$ satisfies
conditions $(\romannumeral1)$ through $(\romannumeral3)$, then there exists a unique $u\in \widetilde{S}^m$ ($u\in \widetilde{S}^m_{nc}$)
such that $[u]_{\la}=v_\lambda$ for each $\la\in [0,1].$
\end{tm}

\begin{tm} \label{fuscbchre}\
Given $F_{USCB} (\mathbb{R}^m)$ ($u\in F_{USCG} (\mathbb{R}^m)$). Then
\\
(\romannumeral1) \ $[u]_\la\in K(\mathbb{R}^m) \cup \{\emptyset\}$ for all $\la\in [0,1]$ ($\la\in (0,1]$);
\\
(\romannumeral2) \ $[u]_\la=\bigcap_{\gamma<\lambda}[u]_\gamma$ for all $\la\in (0,1]$.
\\
(\romannumeral3) \ $[u]_0=\overline{\bigcup_{\gamma>0}[u]_\gamma}$.

Moreover, if the family of sets $\{v_\al:\al\in [0,1]\}$ satisfies
conditions $(\romannumeral1)$ through $(\romannumeral3)$, then there exists a unique $u\in F_{USCB}(\mathbb{R}^m)$ ($u\in F_{USCG}(\mathbb{R}^m)$)
such that $[u]_{\la}=v_\lambda$ for each $\la\in [0,1].$
\end{tm}

\begin{re} \label{nce}
{\rm
Suppose that
 $u\in \widetilde{S}^{m}_{nc}$.
 Denote $\mbox{ker}\, u: =\bigcap_{\alpha \in (0,1]} \mbox{ker}\, [u]_\al$
  (also see \cite{da, qiu}).
  Then,
from Theorems \ref{smre} and \ref{gsmre},
 $u\in S^{m}_{nc}$
 if and only if
 $\mbox{ker}\, u \not= \emptyset$.
}
\end{re}

In this paper, we discuss the relationships and properties of the endograph metric convergence
 and
 the
 $\Gamma$-convergence
on
$F_{USC} (\mathbb{R}^m)$.
These two convergence structures
   can
 be induced from metrics on $F_{USC} (\mathbb{R}^m)$.
In fact,
since ${\rm end}\, u \in    C(  \mathbb{R}^{m+1}  )$ for each $u \in F_{USC} (\mathbb{R}^m)$,
 then from the
conclusions of   Kuratowski convergence
listed
 in
 Section \ref{hfkc},
it follows immediately
that
the
 $\Gamma$-convergence on
$F_{USC} (\mathbb{R}^m)$
 is metrizable.

\section{The relationship between   endograph metric convergence and $\Gamma$-convergence on fuzzy sets \label{rghc}}

We \cite{huang7}
found the interesting fact
that
the $\Gamma$-convergence and the endograph metric convergence coincide on $E^1$.

\begin{pp} \label{hfm} (Remark 1.3 and its proof in Huang and Wu \cite{huang7})
  If $u_n$, $u$ is in $E^1$, $n=1,2,\ldots$, then  $u_n\str{\Gamma}{\longrightarrow}u$
  is
  equivalent to $H_{\rm end} (u_n, u) \to 0$.
\end{pp}

In this section, first we show that the   endograph metric   convergence is stronger than
the
$\Gamma$-convergence
on
 $F_{USC} (\mathbb{R}^m)$.
Then
it is proved
that
an  endograph metric   convergent sequence in $F_{USCG} (\mathbb{R}^m)$ is exactly
a
$\Gamma$-convergent sequence satisfying the condition that
 the union of $\al$-cuts of all its elements is a bounded set in $\mathbb{R}^m$ for each $\al>0$.
As a consequence,
 we
 deduce
that
 a Hausdorff metric convergent
sequence in $K(\mathbb{R}^m) \cup \{\emptyset\}$
is exactly
a Fell topology convergent (Kuratowski convergent) sequence with the union of all its elements being a bounded set in $\mathbb{R}^m$.

In Section \ref{pfuscg}, we will show that the $\Gamma$-convergence and the endograph metric convergence coincide
on
 $F_{USCGCON} (\mathbb{R}^m) \backslash \widehat{\emptyset}$.

\begin{tm} \label{hpkc}
  Suppose that $C$, $C_n$, $n=1,2,\ldots$, are closed sets in $\mathbb{R}^{m}$.
  Then
  $H(   C_n,    C    ) \to 0$ as $n\to \infty$
  implies that
   $\lim_{n\to \infty} C_n (K)=C $.
\end{tm}

\begin{proof}
 \  Suppose that $H( C_n,  C) \to 0$ as $n\to \infty$. If $C=\emptyset$, then
there is an $N$ such that $C_n = \emptyset$ when $n\geq N$.
Thus
$\lim_{n\to \infty} C_n (K)=C $.

If  $C \not=\emptyset$.
Given
$x\in \limsup_{n\rightarrow \infty} C_{n}$.
Then from Lemma \ref{infe},
there exists a subsequence $\{C_{n_i}\}$ of $\{C_n\}$ such that
$ d(x, C_{n_i}) \to 0$.
Note that
$H( C_{n_i},   C) \to 0$,
hence
$d(x,  C)=0$.
It then follows from the closedness of $C$ that $x\in  C$.
So we know
\begin{equation}\label{supc}
  \limsup_{n\rightarrow \infty}   C_{n} \subseteq  C.
\end{equation}
On
 the other hand, given $x\in   C$, then $d(x,  C_n) \to 0$ as $n\to \infty$.
So by Lemma \ref{infe}, we have
 \begin{equation}\label{cinf}
  C\subseteq \liminf_{n\to\infty} C_n.
 \end{equation}
 Combined with \eqref{supc} and \eqref{cinf}, we know that
 $$ C=\lim_{n\to \infty} C_n (K). $$
\end{proof}

\begin{re}{\rm
  From
the basic facts about the Fell topology,
the Kuratowski convergence
and
the Hausdorff metric
stated in Section \ref{bns}, the statement expressed in Theorem \ref{hpkc}
 is a known conclusion.
Here we give the proof for the self-containing of this paper.
One of the referees also kindly reminded us to pay attention to references \cite{matheron,beer, wei} and pointed out that Theorem \ref{hpkc} can be deduced from results in \cite{beer, wei}. However
we are not able to obtain references \cite{matheron,beer, wei}.
}
\end{re}

\begin{tm} \label{uschpm}
   Suppose that $u$, $u_n$, $n=1,2,\ldots$,
   are
   fuzzy sets in $F_{USC}(\mathbb{R}^m)$.
     Then
  $H_{\rm end}   (  u_n,     u   ) \to 0$ as $n\to \infty$
  implies that
   $u_n\str{\Gamma}{\longrightarrow} u $.
\end{tm}

\begin{proof}
 \ Note that ${\rm end}\, u$, ${\rm end}\, u_n$, $n=1,2,\ldots$, are closed sets in $\mathbb{R}^{m+1}$,
hence
the desired result
follows from Theorem \ref{hpkc}.
\end{proof}

\begin{eap}\label{gse}
{\rm
  Consider $u,\ u_n \in F_{USC}(\mathbb{R}^1)$, $n=1,2,\ldots$, defined as follows:
  \[u_n(x)=\left\{
      \begin{array}{ll}
        1, & x=0 \hbox{ or } n, \\
        0, & \hbox{otherwise.}
      \end{array}
    \right.
  \]
\[
u(x)=\left\{
      \begin{array}{ll}
        1, & x=0,  \\
        0, & \hbox{otherwise.}
      \end{array}
    \right.
\]
Then we can see
$u_n \str{\Gamma}{\longrightarrow} u $.
We can also check
that
$H_{\rm end}   (  u_n,     u   ) \geq 1$,
and
hence
$H_{\rm end}   (  u_n,     u   ) \not\to 0$.
}
\end{eap}

Combined with
Theorem \ref{uschpm} and Example \ref{gse},
we know that
the
endograph metric convergence is stronger than
the $\Gamma$-convergence
on
 $F_{USC} (\mathbb{R}^m)$.

For reading and writing convenience, we use the symbol
${\rm end}_\varepsilon\, u$
to denote the set
  ${\rm end}\, u   \cap   \{(y_1, y_2, \ldots, y_{m+1}): y_i\in \mathbb{R}, \ i=1,2,\ldots, m+1, \  y_{m+1}\geq \var   \}$, where $u$ is a fuzzy set on $\mathbb{R}^m$.

\begin{tm}  \label{uscbgche}
  Suppose that $u$, $u_n$,
$n=1,2,\ldots$, are  fuzzy sets in $F_{USCG} (\mathbb{R}^m)$.
   If
\\
  (\romannumeral1)  $u_n\str{\Gamma}{\longrightarrow}u$, and
\\
 (\romannumeral2)  given $\alpha\in (0,1]$,
$ \bigcup_{n=1}^{+\infty}  [u_n]_\al $ is a bounded set in $\mathbb{R}^m $,
\\
  then
   $H_{\rm end}(u_n, u) \to 0$ as $n\to \infty$.
\end{tm}

\begin{proof}
  \  Suppose that $u_n\str{\Gamma}{\longrightarrow}u$, i.e.
${\rm end}\,u=\lim_{n\to \infty}{\rm end}\, u_n (K)$.
We proceed by contradiction to show
$H({\rm end}\,u_n, {\rm end}\,u)\to 0$ as $n\to \infty$.
Note that $ {\rm end}\,u \not= \emptyset$ and ${\rm end}\,u_n \not= \emptyset$, $n=1,2,\ldots$,
hence
if $H({\rm end}\,u_n,  {\rm end}\, u) \not\to 0$,
then
$H^*({\rm end}\,u_n, {\rm end}\,u) \not\to 0$
or
$H^*({\rm end}\,u, {\rm end}\,u_n) \not\to 0$.

If $H^*({\rm end}\,u_n, {\rm end}\,u) \not\to 0$, then there is an $\varepsilon_0 \in (0,1)$
and a subsequence $\{ u_{n_i} \}$ of $\{u_n\}$
such that
$H^*({\rm end}\,u_{n_i}, {\rm end}\,u) > \varepsilon_0$.
Hence there is
a sequence $\{z_{n_i}\}$
satisfying
that
$z_{n_i} \in {\rm end}\,u_{n_i}$
and
\begin{equation}
  d(z_{n_i}, {\rm end}\, u) >  \varepsilon_0.      \label{znie}
\end{equation}
Note
that
$z_{n_i} \in {\rm end}_{\varepsilon_0}\,u_{n_i}$,
and
that
$ \bigcup_{n=1}^{+\infty}  {\rm end}_{\varepsilon_0}\, u_{n} \subset  \bigcup_{n=1}^{+\infty} [u_n]_{\var_0}  \times [\varepsilon_0, 1]$
is a bounded set.
Thus we know
that
$\{z_{n_i} \} $ has a convergent subsequence $ \{z_k^{(1)}, k=1,2,\ldots\}$.
So
\begin{equation*}
  z_0  :=  \lim_{k\to\infty} z_k^{(1)}   \in   \limsup_{n\to\infty} {\rm end}\, u_n =   {\rm end}\, u.
\end{equation*}
This is in contradiction with \eqref{znie}.

If
$H^*({\rm end}\,u,    {\rm end}\,u_n) \not\to 0$, then there exists $\varepsilon_1>0$
and
$\{x_{n_i} \} \subset {\rm end}\,u$
such that
$$d(x_{n_i}, {\rm end}\, u_{n_i}) \geq \varepsilon_1.$$
Notice that $\{ x_{n_i} \}  \in {\rm end}_{\varepsilon_1}\, u$
and
that
${\rm end}_{\varepsilon_1}\, u  $ is compact.
Therefore  $\{ x_{n_i} \}$ has
 a
convergent
subsequence.
With no loss of generality,
we may suppose that
 $\{ x_{n_i} \}$ converges.
Thus
$$x_0 : = \lim_{i\to \infty} x_{n_i} \in {\rm end}_{\varepsilon_1}\, u.$$

So there exists an $M$ such that for all $n_i\geq M$,
$d(x_0,  x_{n_i} ) < \varepsilon_1/2$.
Thus, for all $n_i\geq M$,
$$d(x_0,   {\rm end}\, u_{n_i}) \geq d(x_{n_i},  {\rm end}\, u_{n_i}) - d(x_0,   x_{n_i} ) \geq \varepsilon_1/2.$$
This means that
$x_0   \notin \liminf_{n\to \infty} {\rm end}\, u_n$,
which contradicts the fact that
${\rm end}\, u=\liminf_{n\to\infty} {\rm end}\, u_n$.
\end{proof}

\begin{lm} \label{hbn}
  If $\{u_n, n=1,2,\ldots\}$ is a Cauchy sequence in $(F_{USCG}(\mathbb{R}^m), H_{\rm end} )$,
then
$ \bigcup_{n=1}^{+\infty}  [u_n]_\al $ is a bounded set in $\mathbb{R}^m $ for each $\al\in (0,1]$.
\end{lm}

\begin{proof}
  Note that $\{u_n, n=1,2,\ldots\}$ is a totally bounded set in $(F_{USCG}(\mathbb{R}^m), H_{\rm end} )$. By proceeding according to the necessity part
of Theorem \ref{fusbtoundcn}, we can obtain the desired result.
\end{proof}

The following theorem gives the relationship of
the endograph metric
and
the $\Gamma$-convergence on $F_{USCG} (\mathbb{R}^m)$,
which
improves Theorem \ref{uscbgche}.

\begin{tm} \label{hgce}
   Suppose that $u$, $u_n$,
$n=1,2,\ldots$, are  fuzzy sets in $F_{USCG} (\mathbb{R}^m)$.
  Then $H_{\rm end}(u_n, u) \to 0$ as $n\to \infty$
is equivalent to
\\
  (\romannumeral1)  $u_n\str{\Gamma}{\longrightarrow}u$, and
\\
 (\romannumeral2)  given $\alpha\in (0,1]$,
$ \bigcup_{n=1}^{+\infty}  [u_n]_\al $ is a bounded set in $\mathbb{R}^m $.
   \end{tm}

\begin{proof} \
The desired result follows immediately from Theorems \ref{uschpm} and \ref{uscbgche} and Lemma \ref{hbn}.
\end{proof}

\begin{tl}  \label{gchrbsn}
  Suppose that $u$, $u_n$,
$n=1,2,\ldots$, are  fuzzy sets in $F_{USCG} (\mathbb{R}^m)$.
   If, given $\alpha\in (0,1]$,
$ \bigcup_{n=1}^{+\infty}  [u_n]_\al $ is a bounded set in $\mathbb{R}^m $,
then $u_n\str{\Gamma}{\longrightarrow}u$
is
equivalent
to
   $H_{  {\rm end}    }(u_n, u) \to 0$ as $n\to \infty$.
\end{tl}

\begin{proof}
 \ The desired result follows immediately from Theorem \ref{hgce}.
\end{proof}

As a consequence of the results given above, we discuss the relationship of
the
 Hausdorff metric convergence and Fell topology convergence (Kuratowski convergence) on $K(\mathbb{R}^m) \cup \{ \emptyset \}$.

\begin{lm} \label{rhe}
  Let $D$, $G$, and $D_n$, $n=1,2,\ldots$, be sets in $C(\mathbb{R}^m) \cup  \{\emptyset\}  $. Then
$\widehat{D}$, $\widehat{G}$, and $\widehat{D_n}$, $n=1,2,\ldots$, are fuzzy sets in  $F_{USC} (\mathbb{R}^m)$,
and
\\
(\romannumeral1) \ $\lim_{n\to \infty} D_n (K) =D$ is equivalent to $\widehat{D_n} \str{\Gamma}{\longrightarrow}  \widehat{D}$,
\\
(\romannumeral2) \ $\min \{H(D,G), 1\} = H_{\rm end} (\widehat{D}, \widehat{G})$.
\end{lm}

\begin{proof}
 \ The conclusions are easy to be checked.
\end{proof}

\begin{tl} \label{nfsghe}
   Let $D$, $D_n$,  $n=1,2,\ldots    $ be sets in $K ( \mathbb{R}^m  ) \cup  \{\emptyset\} $.
  Then $H( D_n, D) \to 0$ as $n\to \infty$
is
equivalent
to
\\
(\romannumeral1) \ $\lim_{n\to \infty} D_n (K) = D$, and
\\
(\romannumeral2) \ $\bigcup_{n=1}^{+\infty}  D_n$ is a bounded set in $\mathbb{R}^m$.
\end{tl}

\begin{proof} \
The desired result follows from Theorem \ref{hgce}
and Lemma \ref{rhe}.
\end{proof}

\begin{re} \label{esc}{\rm
  Suppose that $D$, $D_n$,  $n=1,2,\ldots    $ are sets in $C ( \mathbb{R}^m  ) \cup  \{\emptyset\} $,
and
$\lim_{n\to \infty} D_n (K) = D$.
\begin{enumerate}
\renewcommand{\labelenumi}{(\roman{enumi})}
  \item If
$D \not= \emptyset$,
then
there exists an $N$ such that for all $ n\geq N $,
$
  D_n \not= \emptyset.
$
Otherwise $\liminf_{n\to \infty} D_n = \emptyset$, which contradicts $D \not= \emptyset$.

  \item If $D=\emptyset$ and $\bigcup_{n=1}^{+\infty}  D_n$ is a bounded set in $\mathbb{R}^m$,
then
there exists an $N$ such that for all $ n\geq N $,
\begin{equation} \label{dne}
  D_n= \emptyset.
\end{equation}
Otherwise
there exists a strictly increasing sequence $\{n_j: j=1,2,\ldots\}$ such
that
$D_{n_j} \not= \emptyset$. Take $x_j\in D_{n_j}$, $j=1,2,\ldots$.
Note that $\bigcup_{n=1}^{+\infty}  D_n$ is a bounded set in $\mathbb{R}^m$,
hence the sequence
$\{x_j: j=1,2,\ldots\}$
has a cluster point
$x_0$.
Thus
\begin{equation*}
  x_0 \in \limsup_{n \to \infty} D_n  =D.
\end{equation*}
This contradicts $D=\emptyset$.
So \eqref{dne} holds. This implies that $H(D_n, D) \to 0$ as $n\to \infty$.

\end{enumerate}
}

\end{re}

\section{Level characterizations of fuzzy sets} \label{lcfu}

Level characterizations are important properties
 of
 fuzzy sets, which help us to see fuzzy sets more clearly.
 Propositions
  \ref{nr}, a widely used fuzzy numbers representation theorem,
  characterizes
 level cut-sets of fuzzy numbers.
 In
this section, we investigate level characterizations of fuzzy sets.
We
show
that
$D(u)$ (defined in Theorem \ref{fusdctn}) is at most countable
for each
fuzzy set
 $u$ on $\mathbb{R}^m$.
From this fact, it follows that
the platform point (introduced below) of
$u$
is
at most countable when $u\in F_{USC} (\mathbb{R}^m)$.
 By using this result,
we discuss Fell topology continuity and endograph metric continuity of
the cut-set functions of
fuzzy sets in $F_{USC} (\mathbb{R}^m)$ and $F_{USCG} (\mathbb{R}^m)$, respectively.
 These conclusions
are
basis of the results in subsequent sections.

 Suppose that $u$ is a upper semi-continuous fuzzy set on $\mathbb{R}^m$. A number $\al$ in (0,1) is
called a platform point of $u$ if
$\overline{\{u>\al\}}\subsetneqq [u]_\alpha$.
 The set of all
platform points of $u$ is denoted by $P(u)$.

\begin{tm}
\label{fusdctn}
  Let $u$ be a fuzzy set on $\mathbb{R}^m$.
  Then $D(u):=\{\al\in (0,1):  [u]_\al \nsubseteq  \overline{ \{u> \al\}}  \;\}$ is at most countable.
\end{tm}

\begin{proof}
  \ See the Appendix.
\end{proof}

\begin{tm} \label{puc}
  Suppose that $u$ is a fuzzy set in $F_{USC} (\mathbb{R}^m)$.
Then
$P(u)$ is at most countable.
\end{tm}

\begin{proof}
 Notice that
$P(u) \subseteq D(u)$.
By
Theorem \ref{fusdctn},
$D(u)$ is at most countable,
hence we know
$P(u)$ is at most countable.
\end{proof}

Let  $u(\cdot)$ be a function from $[0,1]$ to $  C(\mathbb{R}^m) \cup \{\emptyset\}$.
For convenience,
 we write
 $  \lim_{\beta \to \al} u (\beta) (K)=    B$
($  \lim_{\beta \to \al+} u (\beta) (K)=    B$,
$  \lim_{\beta \to \al-} u (\beta) (K)=     B $)
if for each $\beta_n \to \al$ ($\beta_n \to \al+$,   $\beta_n \to \al-$), it holds that
$  \lim_{n\to \infty} u (\beta_n) (K)=    B $.

Suppose that
$u(\cdot)$
is a function
from $[0,1]$ to $ (  C(\mathbb{R}^m) \cup \{\emptyset\}, \tau_f  )$.
Since the Fell topology
$\tau_f$
on $C(\mathbb{R}^m) \cup \{\emptyset\}$
is metrizable,
we know
that
 $\alpha $ is a continuous point of
 $u(\cdot)$
if and only if
$u(\beta_n)$ converges to $u(\al)$ in $(  C(\mathbb{R}^m) \cup \{\emptyset\}, \tau_f  )$
whenever
$\beta_n \to \al$.
Note that the Kuratowski convergence on $C(\mathbb{R}^m) \cup \{\emptyset\}$
is compatible with
the
Fell topology $\tau_f$ on $C(\mathbb{R}^m) \cup \{\emptyset\}$.
So
$\alpha$ is a continuous point of
 $u(\cdot)$
is equivalent to
$  \lim_{\beta \to \al} u(\beta) (K)=   u(\alpha) $.
Similarly,
$\alpha$ is a left(right)-continuous point of function
 $u(\cdot)$
is equivalent to
$  \lim_{\beta \to \al-} u(\beta) (K)=   u(\alpha) $ ($  \lim_{\beta \to \al+} u(\beta) (K)=   u(\alpha) $).

\begin{tm} \label{fusbpe}
    Let $u$ be a fuzzy set in $F_{USC} (\mathbb{R}^{m})$. Then $u$ naturally derives a function
    $u(\cdot): [0,1] \to (  C(\mathbb{R}^m) \cup \{\emptyset\}, \tau_f  ) $
    which is defined by $u(\alpha) = [u]_\alpha$.
    Suppose that $\al\in (0,1)$. Then the following statements are true.
     \begin{enumerate}
     \renewcommand{\labelenumi}{(\roman{enumi})}
   \item  $  \lim_{\beta \to \al+} [u]_{\beta}(K)=     \overline{   \{u>\alpha\}     } $.

  \item  $u(\cdot)$ is left-continuous at $\al$, i.e.
  $  \lim_{\beta \to \al-} [u]_{\beta}(K) =  [u]_\al $.

  \item  $\alpha \in P(u)$ if and only if
 $\alpha$ is a discontinuous point of function
 $u(\cdot)$.

\end{enumerate}

\end{tm}

\begin{proof} \  Note that
  $[u]_\al \subseteq [u]_\beta$    whenever $\alpha > \beta$.
So
$u(\cdot): (0,1] \to    C(\mathbb{R}^m) \cup \{\emptyset\}$ is a monotone function in this sense.
Let
$\al \in (0,1)$.
If $ \overline{   \{u>\alpha\}     }=\emptyset$,
then
 $  \lim_{\beta \to \al+} [u]_{\beta}(K)=     \overline{   \{u>\alpha\}     } $ holds obviously.

If $ \overline{   \{u>\alpha\}     } \not= \emptyset$.
Take
 $\beta_n \to \al+$, $n=1,2,\ldots$.
 Note that $\beta_n >\al$, hence
\begin{equation} \label{supcon}
  \limsup_{n\to\infty} [u]_{\beta_n} \subseteq   \overline{   \{u>\alpha\}     } .
\end{equation}
On the other hand, given $x\in \overline{   \{u>\alpha\}     } $, if $u(x) >\al$, then,
clearly, $x\in  \liminf_{n\to\infty} [u]_{\beta_n} $.
If $u(x)=\al$,
then there exists an sequence $\{x_m\} \subset   \{u>\alpha\}  $
such that $x=\lim_{m\to\infty } x_m$.
Hence
$\lim_{m\to\infty }  u( x_m ) = \al$.
With no loss of generality,
we can suppose that
$\beta_1= \max\{\beta_n: n=1,2,\ldots\}$
and
$u(x_1)=\beta_1$.
Set
\[
y_n= x_m,  \ \hbox{where}\  m=\max\{ l: x_l \in [u]_{\beta_n}     \}.
 \]
Then
$y_n \in [u]_{\beta_n}$.
It can be checked that
$\lim_{n\to \infty} y_n = x$. From
the arbitrariness of $x \in  \overline{   \{u>\alpha\}     }$, we thus know
\begin{equation} \label{infcon}
  \overline{   \{u>\alpha\}     }  \subset   \liminf_{n\to\infty} [u]_{\beta_n}.
\end{equation}
Combined with \eqref{supcon}
and
\eqref{infcon},
we obtain
\begin{equation*}
   \lim_{n \to \infty} [u]_{\beta_n}(K)    =       \overline{   \{u>\alpha\}     }.
\end{equation*}
Since $\{\beta_n\}$ is an arbitrary sequence satisfying $\beta_n \to \alpha+$,
this implies that
  $  \lim_{\beta \to \al+} [u]_{\beta}(K)=     \overline{   \{u>\alpha\}     }.$
So statement (\romannumeral1) is proved.

Given $\beta_n  \to \al-$, $n=1,2,\ldots$. Then
\begin{equation*}
[u]_\al  \subset   \liminf_{n\to\infty} [u]_{\beta_n}.
\end{equation*}
Since $u$ is upper semi-continuous, it can be checked that
\begin{equation*}
  [u]_\al  \supset   \limsup_{n\to\infty} [u]_{\beta_n}.
\end{equation*}
Thus we know
\begin{equation*}
  [u]_\al  =  \lim_{n\to\infty} [u]_{\beta_n} (K).
\end{equation*}
So
\begin{equation*}
  [u]_\al  =  \lim_{\beta \to \al-} [u]_{\beta} (K).
\end{equation*}
This is statement (\romannumeral2).

$\al\notin P(u) $ is equivalent to $\overline{\{ u> \al\}} =  [u]_\al$.
From statements (\romannumeral1)
and
(\romannumeral2),
$\overline{\{ u> \al\}} =  [u]_\al$
if and only if
\begin{equation*}
  [u]_\al=  \lim_{\beta \to \al} [u]_{\beta} (K),
\end{equation*}
which means
that
$\al$ is a continuous point of $u(\cdot)$.
So
we obtain statement (\romannumeral3).
\end{proof}

Suppose that $u\in  F_{USCG}(\mathbb{R}^m)$ and $u\not= \widehat{\emptyset}$.
Then
we can check that
there
exists a unique \bm{$\lambda_u} \in (0,1]$
such that $\beta \in [0, \lambda_u]$
 if and only if
 $[u]_\beta \not= \emptyset$.

\begin{tm} \label{smncbpe}
    Suppose that $u\in  F_{USCG}(\mathbb{R}^m)$. Then $u$ naturally derives a function
    $u(\cdot): (0,   1] \to (K(\mathbb{R}^m)  \cup \{\emptyset\}, H)$
    which is defined by $u(\alpha) = [u]_\alpha$. Let $\al\in (0,1)$. Then the following statements are true.
     \begin{enumerate}
     \renewcommand{\labelenumi}{(\roman{enumi})}
   \item  $ H(   [u]_{\beta},     \overline{   \{u>\alpha\}     }      ) \to 0 $ as $\beta \to \al+$.

  \item  $u(\cdot)$ is left-continuous at $\al$, i.e. $ H(   [u]_{\beta},    [u]_\al     ) \to 0 $ as $\beta \to \al-$.

  \item  $\alpha \in P(u)$ if and only if
 $\alpha$ is a discontinuous point of function
 $u(\cdot)$.

\end{enumerate}

\end{tm}

\begin{proof}
 \ Suppose that $u= \widehat{\emptyset}$.
Then
$[u]_\al = \overline{\{u>\alpha\}} = \emptyset$
when
$\al \in [0,1]$.
So
$u(\cdot)$ is continuous on $(0,1)$ and $P(u) = \emptyset$.
Clearly,
statements (\romannumeral1), (\romannumeral2), and (\romannumeral3)
hold.

Suppose that  $u \not= \widehat{\emptyset}$.
Note that
  $[u]_\al \subseteq [u]_\beta$    whenever $\alpha > \beta$.
So
$u(\cdot): (0,1] \to K(\mathbb{R}^m) \cup \{\emptyset\}    $ is a monotone function.
To show statements (\romannumeral1), (\romannumeral2) and (\romannumeral3)
are true for any $\al\in (0,1)$,
we divide
the remainder proof into three cases.

Case (\uppercase\expandafter{\romannumeral1}) \ $\al\in (\lambda_u, 1)$.

In this case, $[u]_\al = \overline{\{u>\alpha\}} = \emptyset$.
This implies
that
statements (\romannumeral1), (\romannumeral2), and (\romannumeral3)
hold when
$\al \in (\lambda_u, 1)$, and that $P(u) \cap  (\lambda_u, 1)= \emptyset$.

Case (\uppercase\expandafter{\romannumeral2}) \  $\alpha \in (0,\lambda_u)$.

Given $\beta_n \to \alpha+$.
Then
$
\overline{\{u>\alpha\}}  =\overline{\bigcup_{\beta>\alpha} [u]_{\beta}    } =\overline{ \bigcup_{n=1}^{+\infty} [u]_{\beta_n}}. $
So, by Proposition \ref{mce},
\begin{equation*}\label{scn}
  H(   [u]_{\beta_n},     \overline{   \{u>\alpha\}     }      ) \to 0 \ \hbox{as} \  n\to \infty.
\end{equation*}
The monotonicity of $u(\cdot)$ yields
$$
 H(   [u]_{\beta},     \overline{   \{u>\alpha\}     }      ) \to 0     \ \hbox{as} \      \beta \to \al+.
 $$
 Statement (\romannumeral1) is proved.

Note that
$
 [u]_\al=\bigcap_{\beta< \alpha} [u]_{\beta}. $
 By
Proposition \ref{mce},
we can prove that
 $$
 H(   [u]_{\beta},    [u]_\al     ) \to 0     \ \hbox{as} \      \beta \to \al-.
 $$
This is statement (\romannumeral2).

$\al$ is a platform point of
$u$, i.e.,
$\overline{\{u>\al\}}\subsetneqq [u]_\alpha$, which is equivalent to
$$H(\overline{\{u>\al\}},  [u]_\alpha) >0 .$$
From statements
(\romannumeral1) and (\romannumeral2),
the inequality above holds if and only if
$\al$ is a
discontinuous point of function
 $u(\cdot)$.
Hence statement (\romannumeral3) holds.

Case (\uppercase\expandafter{\romannumeral3}) \  $\al=\lambda_u$.

Note that $\overline{ \{u> \lambda_u \} }= \emptyset$. So statement (\romannumeral1)
holds.
By Proposition \ref{mce}, we can show statement (\romannumeral2).
Since $[u]_{\lambda_u} \not= \emptyset$,
we
know
that
$\lambda_u \in P(u)$ and $\lambda_u$ is a discontinuous point of $u(\cdot)$.
This
means
that
statement (\romannumeral3) is true when $\al = \lambda_u$.
\end{proof}

\begin{re} \label{lamc}{\rm
  From the proof of Theorem \ref{smncbpe}, we know that if $u\in F_{USCG}(\mathbb{R}^m) \backslash \widehat{\emptyset}$,
  then
$\lambda_u \in (0,1)$ if and only if $\lambda_u \in  P(u)$.}
\end{re}

\section{Level characterizations of $\Gamma$-convergence and endograph metric convergence \label{slc}}

Huang and Wu \cite{huang7} and Fan \cite{fan3} presented level decomposition properties of $\Gamma$-convergence and endograph metric convergence on one-dimensional fuzzy numbers, independently.

\begin{pp}\label{hlc} (Theorems 2.7 and 2.8 and Remark 2.9 in Huang and Wu \cite{huang7}) The following are equivalent statements
on
fuzzy numbers $u_n$, $u$ in $E^1_{nc}$, $n=1,2,\ldots$.
\\
(\romannumeral1)\ $u_n\str{\Gamma}{\longrightarrow}u$.
\\
  (\romannumeral2)\ $H([u_n]_\al,  [u]_\al) \to 0$ holds a.e. on $\al\in (0,1)$.
\\
(\romannumeral3)\  $H([u_n]_\al,  [u]_\al) \to 0$ holds when
$\al\in (0,1)\backslash P(u)$.
  \\
(\romannumeral4)\ $H([u_n]_\al,  [u]_\al) \to 0$ holds when
$\al\in P$, where $P$ is a dense subset of $(0,1) \backslash P(u)$.
  \\
  (\romannumeral5
  )\ $H([u_n]_\al,  [u]_\al) \to 0$ holds when $\al\in P$, where $P$ is a countable dense subset of $(0,1) \backslash P(u)$.

\end{pp}

\begin{pp}\label{flc} (Theorem 3 and the proof of Lemma 2 in Fan \cite{fan3}) Let $u_n$, $u$, $n=1,2,\ldots$, be fuzzy numbers in $E^1$.
Then the following conditions are equivalent:
\\
(\romannumeral1)\ $H_{\rm end}(u_n,  u) \to 0$.
\\
(\romannumeral2)\ $H([u_n]_\al,  [u]_\al) \to 0$ for $\al\in [0,1]$ almost everywhere.
\\
(\romannumeral3)\  $H([u_n]_\al,  [u]_\al) \to 0$ for
$\al\in (0,1)\backslash P(u)$.
\\
(\romannumeral4)\ $H([u_n]_\al,  [u]_\al) \to 0$ for some dense subset of $[0,1]$.
\\
(\romannumeral5)\  $H([u_n]_\al,  [u]_\al) \to 0$ for some countably dense subset of $[0,1]$.

\end{pp}

We \cite{huang7} in fact proved
the following
statement A (see the proofs of Theorems 2.7 and 2.8, and Remark 2.9 in \cite{huang7}). Since it is pointed out the
 fact that $H([a_n,b_n], [a,b]) \to 0$ is equivalent to
  $[a,b]=\lim_{n\to\infty} [a_n, b_n]  (K)$ in Lemma 2.5 of \cite{huang7} (now we know that this is an already known fact, see Propositions \ref{hms} and \ref{klhe}),
then we obtained
Proposition \ref{hlc}.

\begin{description}
  \item[Statement A] The following are equivalent statements
on
fuzzy numbers $u_n$, $u$ in $E^1_{nc}$, $n=1,2,\ldots$.
\\
(\romannumeral1)\ $u_n\str{\Gamma}{\longrightarrow}u$.
\\
 (\romannumeral2)\ $\lim_{n\to \infty} [u_n]_\alpha (K)= [u]_\al$ holds a.e. on $\al\in (0,1)$.
\\
(\romannumeral3)\  $\lim_{n\to \infty} [u_n]_\alpha (K)= [u]_\al$ when
$\al\in (0,1)\backslash P(u)$.
 \\
(\romannumeral4)\ $\lim_{n\to \infty} [u_n]_\alpha (K)= [u]_\al$ holds when
$\al\in P$, where $P$ is a dense subset of $(0,1) \backslash P(u)$.
  \\
  (\romannumeral5)\ $\lim_{n\to \infty} [u_n]_\alpha (K)= [u]_\al$ holds when $\al\in P$, where $P$ is a countable dense subset of $(0,1) \backslash P(u)$.

\end{description}

Proposition \ref{hlc} and statement A are in the setting of   $E^{1}_{nc}$,
Proposition 6.2 is in the setting of
$E^1$.
Trutschnig \cite{trutschnig} have proven an important fact that $P(u)$ is countable when $u\in E^m_{nc}$.
Using this fact and proceeding according to \cite{huang7},
it then follows
 that
 statement A still hold
in the setting of $E^m_{nc}$.
Many results in previous works can be deduced from the statement A in the setting of $E^m_{nc}$,
some of which are listed in Remark \ref{ems}.

\begin{re}{\rm

From the results in \cite{huang7} (Proposition \ref{hfm} and Lemma 2.5 of \cite{huang7}), it follows immediately that
Propositions \ref{hlc} and \ref{flc} are the same conclusion when restricted to     $E^{1}$, and that
Proposition \ref{hlc} and statement A are
the same conclusion.

Furthermore, in the sequel, we will show that Propositions \ref{hlc} and \ref{flc} and statement A
  still hold and are
  just the same conclusion in the setting of $\widetilde{S}^{m}_{nc}$, and so are the same conclusion in the setting of any common fuzzy set class, such as
 $
E^m $, $  E^{m}_{nc}$,
$
  S^m $, $   S^{m}_{nc}$ and
$\widetilde{S}^{m} $ (see Theorem \ref{cgme}).
}
\end{re}

It is natural to consider whether these level characterizations
still hold on
general fuzzy sets whose $\al$-cuts do not have assumptions of normality, convexity, star-shapedness,
or
 even connectedness.
In Subsection \ref{lcge},
we give level characterizations of the $\Gamma$-convergence
on
$F_{USC} (\mathbb{R}^m)$.
Based on
these conclusions and
the
results in
Section \ref{rghc},
we
obtain
level characterizations of the endograph metric convergence
on
$F_{USCG} (\mathbb{R}^m)$.
In Subsection \ref{regp},
by using
these level characterizations,
we
 consider the relationships among $d_p$ metric, endograph metric, and $\Gamma$-convergence
on $F_{USC} ( \mathbb{R}^m )$.

\subsection{Level characterizations of $\Gamma$-convergence and endograph metric convergence \label{lcge}}

\begin{pp}\cite{rojas} \label{Gcln}
   Suppose that $u$, $u_n$, $n=1,2,\ldots$, are fuzzy sets
  in $F_{USC} ( \mathbb{R}^m )$.
  Then
 $u_n\str{\Gamma}{\longrightarrow}u$ iff for
all $\al\in (0,1],$
\begin{equation} \label{gcln}
   \{u>\alpha\}
   \subseteq
\liminf_{n\to \infty}[u_n]_{\alpha}
\subseteq
\limsup_{n\to \infty}[u_n]_\alpha
\subseteq
 [u]_\alpha.
\end{equation}
\end{pp}

\begin{re}
{\rm
Rojas-Medar and Rom\'{a}n-Flores (Proposition 3.5 in \cite{rojas})
presented
the
statement in Proposition \ref{Gcln} when $u$, $u_n$, $n=1,2,\ldots$, are fuzzy sets in $E^m$.
It can be checked
that
this conclusion also holds when $u$, $u_n$, $n=1,2,\ldots$, are fuzzy sets in $F_{USC} ( \mathbb{R}^m )$.
}
\end{re}

\begin{tm} \label{Gclnre}
   Suppose that $u$, $u_n$, $n=1,2,\ldots$, are fuzzy sets
  in $F_{USC} ( \mathbb{R}^m )$.
  Then
 $u_n\str{\Gamma}{\longrightarrow}u$ iff for
all $\al\in (0,1],$
\begin{equation*}
  \overline{     \{u>\alpha\}     }
   \subseteq
\liminf_{n\to \infty}[u_n]_{\alpha}
\subseteq
 \limsup_{n\to \infty}[u_n]_\alpha
\subseteq
 [u]_\alpha.
\end{equation*}
\end{tm}

\begin{proof}
  \ The desired result follows immediately from Theorem \ref{infc} and Proposition \ref{Gcln}.
\end{proof}

The following theorem shows that the level decomposition properties of $\Gamma$-convergence on $F_{USC} (\mathbb{R}^m)$.
Its proof is a modification of the proof of Theorem 2.7 in \cite{huang7}.

\begin{tm} \label{ltgc}
 Suppose that $ u$, $u_n$,
$n=1,2,\ldots$, are  fuzzy sets in $F_{USC} ( \mathbb{R}^m )$.
Then the
following statements are equivalent.
\\
(\romannumeral1)\ $u_n\str{\Gamma}{\longrightarrow}u$.
\\
(\romannumeral2)\ $\lim_{n\to \infty} [u_n]_\alpha (K)= [u]_\al$ holds a.e. on $\al\in (0,1)$.
\\
(\romannumeral3)\  $\lim_{n\to \infty} [u_n]_\alpha (K)= [u]_\al$ for all
$\al\in (0,1)\backslash P(u).$
\\
 (\romannumeral4)\  $\lim_{n\to \infty} [u_n]_\alpha (K)= [u]_\al$ holds when
$\al\in P$, where $P$ is a dense subset of $(0,1) \backslash P(u)$.
  \\
  (\romannumeral5
  )\  $\lim_{n\to \infty} [u_n]_\alpha (K)= [u]_\al$ holds when $\al\in P$, where $P$ is a countable dense subset of $(0,1) \backslash P(u)$.
\end{tm}

\begin{proof}
 \  Suppose that $u_n\str{\Gamma}{\longrightarrow}u$. By Theorem \ref{Gclnre}, we know
for all
$\al\in (0,1)\backslash P(u)$,
$\lim_{n\to \infty}   [u_n]_\al = [u]_\al (K)$.
Hence
statement (\romannumeral1) implies statement (\romannumeral3).

From Theorem \ref{puc}, we know $P(u)$ is at most countable.
Thus
statement  (\romannumeral3) implies statement  (\romannumeral2).

Suppose that statement (\romannumeral2) holds.
Denote the set
$\{\al\in (0,1]:  \lim_{n\to \infty} [u_n]_\alpha (K)= [u]_\al   \}$
by $C(u)$. So, for each $\gamma\in C(u)$,
$$
[u]_\gamma=\liminf_{n\to \infty} [u_n]_\gamma = \limsup _{n\to \infty} [u_n]_\gamma.
$$
 If
  $\al\in (0,1)$,
  then
   there exists $\{\al_i \}, \{\beta_i \}\subseteq
C(u)$ such that ${\alpha_i \to \al+},\ \beta_i \to
\alpha-$. Thus,
\begin{gather}
  \{u>\alpha\} = \bigcup_{i}\{u>\alpha_i \}  \subseteq
\bigcup_{i}\liminf_{n\to \infty}   [u_n]_{\alpha_i}
\subseteq \liminf_{n\to \infty}    [u_n]_\alpha,    \label{setinfe}
\\
  \limsup_{n\to \infty}[u_n]_\alpha \subseteq
\bigcap_{i}\limsup_{n\to \infty} [u_n]_{\beta_i}  =
\bigcap_{i}[u]_{\beta_i} = [u]_\alpha.   \label{setsupe}
\end{gather}
If $\al=1$,
we
 can find a sequence $\{ \gamma_i \} \subset
C(u), \gamma_i \to 1-$, and then
\begin{equation}
  \limsup_{n\to \infty}    [u_n]_1
   \subseteq
    \bigcap_{i}    \limsup_{n\to \infty}     [u_n]_{\gamma_i}
    =
\bigcap_{i}[u]_{\gamma_i} = [u]_1. \label{1setre}
\end{equation}
Combined
with
 \eqref{setinfe}, \eqref{setsupe} and \eqref{1setre},
 we can see that \eqref{gcln} holds for all $\al \in (0,1]$.
 Thus, by Proposition \ref{Gcln},
 $u_n\str{\Gamma}{\longrightarrow}u$.
 So
 statement (\romannumeral2)
implies statement (\romannumeral1).

In a similar manner, we can show that statements    (\romannumeral1),  (\romannumeral4) and (\romannumeral5)
are
equivalent   to each other.
\end{proof}

\begin{tm}
  Suppose that $u_n$,
$n=1,2,\ldots$, are  fuzzy sets in $F_{USC} (\mathbb{R}^m)$.
Then
the
following two statements are equivalent.
\\
(\romannumeral1)\ $u_n\str{\Gamma}{\longrightarrow} \widehat{\emptyset}$.
\\
(\romannumeral2)\  $\lim_{n\to \infty} [u_n]_\al (K) = [\widehat{\emptyset}]_\al   = \emptyset$ for all $\al\in (0,1]$.
\end{tm}

\begin{proof}
  \ The desired conclusion follows immediately from Theorem \ref{ltgc}.
\end{proof}

\begin{tm} \label{hpe}
 Suppose that $ u$, $u_n$,
$n=1,2,\ldots$, are  fuzzy sets in $F_{USCG} ( \mathbb{R}^m )$.
Then
the
following statements are equivalent.
\\
(\romannumeral1)\ $H_{\rm end} (u_n, u) \to 0$ as $n\to \infty$.
\\
(\romannumeral2)\ $H([u_n]_\al,  [u]_\al) \to 0$ holds a.e. on $\al\in (0,1)$.
\\
(\romannumeral3)\  $H([u_n]_\al,  [u]_\al) \to 0$ holds when
$\al\in (0,1)\backslash P(u)$.
\\
 (\romannumeral4)\   $H([u_n]_\al,  [u]_\al) \to 0$  holds when
$\al\in P$, where $P$ is a dense subset of $(0,1) \backslash P(u)$.
  \\
  (\romannumeral5
  )\  $H([u_n]_\al,  [u]_\al) \to 0$  holds when $\al\in P$, where $P$ is a countable dense subset of $(0,1) \backslash P(u)$.

\end{tm}

\begin{proof}
  \  Suppose that $H_{\rm end} (u_n, u) \to 0$ as $n\to \infty$. Then, by
 Theorems \ref{ltgc} and \ref{uschpm},
$\lim_{n\to \infty} [u_n]_\alpha (K)= [u]_\al$ for all
$\al\in (0,1)\backslash P(u)$.
By Lemma \ref{hbn}
and
Corollary \ref{nfsghe},
we know
$H([u_n]_\al,  [u]_\al) \to 0$ holds when
$\al\in (0,1)\backslash P(u)$.
So
(\romannumeral1) implies (\romannumeral3).

Suppose
that
 $H([u_n]_\al,  [u]_\al) \to 0$ when
$\al\in (0,1)\backslash P(u)$.
By
 Theorems    \ref{hpkc}     and \ref{ltgc},
we know
$u_n \str{\Gamma}{\longrightarrow} u$.
Notice that,
given $\alpha\in (0,1]$,
$ \bigcup_{n=1}^{+\infty}  [u_n]_\al $ is a bounded set in $\mathbb{R}^m $.
It
thus
follows from Theorem \ref{uscbgche}
that
 $H_{\rm end} (u_n, u) \to 0$ as $n\to \infty$.
So
(\romannumeral3) implies (\romannumeral1).

Thus we know that (\romannumeral1) is equivalent to (\romannumeral3).
In a similar manner, we
can
show that     (\romannumeral1)
is equivalent
to
 (\romannumeral2), (\romannumeral4) or (\romannumeral5).
\end{proof}

\subsection{An application of the level characterizations---relationships among $L_p$-metric $d_p$, endograph metric $H_{\rm end}$ and $\Gamma$-convergence on $F_{USC} (\mathbb{R}^m)$ \label{regp}}

We consider relationships   among $d_p$, $H_{\rm end}$ and $\Gamma$-convergence on $F_{USC} (\mathbb{R}^m)$.
The
results are summarized at the end of this subsection.

\begin{tm} \label{hdp}
  Let $u_n, u\in F_{USCG} ( \mathbb{R}^m )$,
$n=1,2,\ldots$. If $\bigcup_{n=1}^{+\infty} [u_n]_0 $ is bounded in $\mathbb{R}^m$ and
$H_{\rm end} (u_n, u)  \to 0$, then
$d_p (u_n, u) \to 0$.
\end{tm}

\begin{proof}
   \ We proceed as in the proof of Theorem 3.2 in \cite{huang7}.
From Theorem \ref{hpe},
 $H([u_n]_\al,  [u]_\al) \to 0$ holds a.e. on $\al\in (0,1)$. Combining with
 the
 boundedness of
$ \bigcup_{n=1}^{+\infty}  [u_n]_0 $ and applying the Lebesgue's
Dominated Convergence Theorem, we can conclude
$d_p (u_n, u) \to 0$.
\end{proof}

\begin{re}
  Conclusion in Theorem \ref{hdp} has already been shown in the setting of
 $E^1_{nc}$, the set of one-dimensional noncompact fuzzy numbers, in \cite{huang7, fan3}.
\end{re}

\begin{tl} \label{gdp}
  Let $u_n, u\in F_{USCG} ( \mathbb{R}^m )$,
$n=1,2,\ldots$. If $\bigcup_{n=1}^{+\infty} [u_n]_0 $ is bounded in $\mathbb{R}^m$ and
$u_n \str{\Gamma}{\longrightarrow} u$, then
$d_p (u_n, u) \to 0$.
\end{tl}

\begin{proof}
 The desired result follows immediately from Theorems \ref{uscbgche} and \ref{hdp}.
 \end{proof}

\begin{lm} \label{dpfm}
Let $u_n, u\in F_{USC} ( \mathbb{R}^m )$,
$n=1,2,\ldots$. If $d_p (u_n, u) \to 0$, then there
exists a subsequence $\{u_{n_i}\}$ of $\{u_n\}$ such that
 $H_{\rm end} (u_{n_i},  u) \to 0$.
\end{lm}

\begin{proof}
   \ We proceed as in the proof of Theorem 3.3 in \cite{huang7}.
   Let
$d_p (u_n, u) \to 0$. Then $\{H([u_n]_\al, [u]_\al)\}$ converges in measure to $0$
on $\al\in [0,1]$.
 As a consequence of the F.Riesz
Theorem, there exists a subsequence $\{u_{n_i}\}$ such that
$\{H([u_{n_i}]_\al, [u]_\al)\}$ converges to $0$ a.e. on $\al\in [0,1]$. Hence using
Theorem \ref{hpe} $H_{\rm end} (u_{n_i},  u) \to 0$.
\end{proof}

\begin{re}
  We    \cite{huang7} proved Lemma \ref{dpfm} in the setting of $E^1_{nc}$.
\end{re}

\begin{tm}\label{dpf}
Let $u_n, u\in F_{USC} ( \mathbb{R}^m )$,
$n=1,2,\ldots$. If $d_p (u_n, u) \to 0$, then
 $H_{\rm end} (u_n,  u) \to 0$.
\end{tm}

\begin{proof} \ We proceed by contradiction. Let $\{u_{n}^{(1)}\}$ be a subsequence of
$\{u_n\}$
such that  $H_{\rm end} (u_n^{(1)},  u) \geq \varepsilon$,
where
$\varepsilon$ is a positive number. Since $d_p (u_n^{(1)}, u) \to 0$,
  then from Lemma \ref{dpfm}, there exists a
  subsequence  $\{u_n^{(2)} \}$ of
$\{u_n^{(1)}\}$
such that $H_{\rm end} (u_n^{(2)},  u) \to 0$, which contradicts $H_{\rm end} (u_n^{(1)},  u) \geq \varepsilon$.
\end{proof}

\begin{re}
  Fan (Remark 2 in \cite{fan3}) claimed that the conclusion of Theorem \ref{dpf} holds in the setting of $E^1$.
\end{re}

\begin{eap} \label{dphe}
  Set $u$, $u_n\in E^1_{nc}$, $n=1,2,\ldots$, as follows:
  \begin{gather*}
   u(x) =\left\{
       \begin{array}{ll}
         1, & \ x=0;\\
         0, & \ x\not= 0,
       \end{array}
     \right.
     \\
   u_n(x)= \left\{
      \begin{array}{ll}
       1, & \ x\in [0, \frac{1}{n}]; \\
        \frac{1}n{}, &  \ x\in [ \frac{1}{n}, n]; \\
        0, & \ x\notin [0,n].
      \end{array}
    \right.
  \end{gather*}
Then $H_{\rm end} (u_n, u) \to 0$    whereas      $d_p(u_n,u) \not\to 0$.

\end{eap}

\begin{re}{\rm
  The results in this section   can imply some results in previous work. For instance, see Remark \ref{ems}.
}
\end{re}

From Theorems \ref{uschpm} and \ref{dpf}, Examples  \ref{gse} and \ref{dphe}, and Corollary \ref{gdp}, we know the following fact.

Let $u, u_n$ in $ F_{USC} ( \mathbb{R}^m )$, $n=1,2,\ldots$. Then
\begin{gather*}
 u_n \str {d_p}{\longrightarrow}u
\Rrightarrow
     u_n   \str {H_{\rm end}}{\longrightarrow}   u
\Rrightarrow
 u_n   \str {\Gamma}{\longrightarrow}     u,
\\
    u_n  \str {\Gamma}{\longrightarrow}  u   +     \hbox{the  boundedness  of} \bigcup_{n=1}^{+\infty} [u_n]_0
\Rrightarrow
u_n \str {d_p}{\longrightarrow}   u,
\end{gather*}
where
$P \Rrightarrow Q$ means that $P$ implies $Q$ however $Q$ can not imply $P$.

\section{Properties of $(F_{USCG}(\mathbb{R}^m), H_{\rm end})$ \label{uscghp}} \label{uscgpre}

  In this section, we give characterizations of relative compactness, total boundedness and compactness in
$(F_{USCG}(\mathbb{R}^m), H_{\rm end})$.
It is found
that
a set is relatively compact if and only if it is totally bounded in
$(F_{USCG}(\mathbb{R}^m), H_{\rm end})$.
This fact
yields
that
$(F_{USCG}(\mathbb{R}^m), H_{\rm end})$ is a complete space.
At the end of this section,
we discuss
the connections
between
the characterizations presented in this section
and
the relationships of endograph metric convergence and $\Gamma$-convergence
found
 in
Section \ref{rghc}.

\begin{tm} \label{fuscbrcn}

Let $U$ be a subset of $F_{USCG}(\mathbb{R}^m)$.
Then
  $U$ is a relatively compact set in $(F_{USCG}(\mathbb{R}^m), H_{\rm end})$
  if and only if
 $$ \bm{U(\al)} := \bigcup_{      u\in U   }    [u]_\al $$ is a bounded set in $\mathbb{R}^m $ when $\al \in (0,1]$.
\end{tm}

\begin{proof}
  \   \textbf{ \emph{Necessity.}} \  Note that a relatively compact set is a totally bounded set. By proceeding according to the necessity part
of Theorem \ref{fusbtoundcn}, we
can
prove that
$U(\al)$ is a bounded set in $\mathbb{R}^m $ when $\al \in (0,1]$.

\textbf{\emph{Sufficiency.}} \ Suppose that, for each  $\al\in (0,1]$, $ U(\al) $ is a bounded set in $\mathbb{R}^m $.
Let $\{u_n\}$ be an arbitrarily chosen sequence in $U$.
To show that $U$ is a relatively compact set,
we only need to prove
that
$\{u_n\}$
has a subsequence $\{v_n\}$
which converges to $v$ in $F_{USCG} (\mathbb{R}^m)$,
i.e.,
$H_{\rm end} (v_n, v) \to 0$.
The proof is split into two steps.

\textbf{Step 1.}  \ Construct $\{v_n\}$ and find $v$.

First we affirm the following statement.
\begin{description}
  \item[] \ Given a sequence $\{w_n : n=1,2,\ldots\}$ in $U$ and $\al\in (0,1]$.
Then the corresponding sequence
$\{[w_n]_\al: n=1,2,\ldots\}$
has a convergent subsequence in $(  K(\mathbb{R}^m) \cup \{\emptyset\},  H )$.
\end{description}
In fact,
if
there exists an $N$ such that
$[w_{n}]_\al \equiv \emptyset$
when $n \geq N$,
then, clearly,
$\{ [w_{n}]_\al \} $ converges to $\emptyset$.
Otherwise
$\{[w_n]_\al: n=1,2,\ldots\}$ has a subsequence $\{[w_{n_j}]_\al: j=1,2,\ldots\}$ which is contained in $K(\mathbb{R}^m)$.
Since
 $ U(\al) $ is a bounded set in $\mathbb{R}^m $,
 by Proposition \ref{sca},
$\{ [w_{n_j}]_\alpha: j=1,2,\ldots \}$ is a relatively compact set in
$(K(\mathbb{R}^m), H)$.
Hence
$\{ [w_{n_j}]_\alpha: j=1,2,\ldots \}$ has a convergent subsequence in $(K(\mathbb{R}^m), H)$.

Now,
arrange
all rational numbers in
$(0,1]$
into a sequence
 $q_1,q_2,\ldots, q_n,\ldots$.
Then
 $\{u_n\}$
has a subsequence $\{u_n^{(1)}\}$ such that
$\{[u_n^{(1)}]_{q_1}\}$
converges to $u_{q_1} \in K(\mathbb{R}^m) \cup \{\emptyset\}$,
i.e. $H(   [u_n^{(1)}]_{q_1},     u_{q_1}   )  \to 0$.
If   $\{u_n^{(1)}\}, \ldots, \{u_n^{(k)}\}$ have been chosen,
then
we
can
choose
a subsequence $\{    u_n^{  (k+1)  }    \}$
of
$\{   u_n^ { (k) }   \}$
such that
 $\{[u_n^{(k+1)}]_{   q_{k+1}   }\}$
converges to
$u_{   q_{k+1}    }
 \in
 K (\mathbb{R}^m)\cup \{\emptyset\}
 $.
Thus we obtain $\{ u_{q_k}, k=1,2,\ldots \} $
which
has the following properties:
\begin{enumerate}\renewcommand{\labelenumi}{(\alph{enumi})}
  \item $ u_{q_m} \subseteq   u_{q_l}$  whenever $q_m > q_l $;
  \item   $   u_{q_k}  \in  K(\mathbb{R}^m) \cup \{\emptyset\}$ for each $k=1,2,\ldots$.
\end{enumerate}
In fact,
since
$ H( [u_n^{(k)}]_{ q_{k} },    u_{q_k}) \to 0$ as $n\to \infty$ for $k=1,2,\ldots$,
by Theorem \ref{hpkc}, we
get
property (a).
From the definition of  $  \{ u_{q_k} \} $, we know that property (b) holds.

Put $v_n=\{    u_n^{  (n) }    \}$ for $n=1,2,\ldots$.
Then $\{v_n\}$
is
a subsequence of
$\{u_n\}$
and
\begin{equation}\label{fusbqpc}
  H([v_n]_{q_k},   u_{q_k})  \to 0 \ \ \ \mbox{as} \ \ \ n\to \infty
\end{equation}
for   $k=1,2,\ldots$.
Define $\{v_\al:  \al\in (0,1] \} $ as follows:
\[
v_\alpha
=
 \bigcap_{q_k  <   \alpha}  u_{q_k},  \   \hbox{for each} \  \alpha \in (0,1].
\]
Then      $\{v_\al : \al\in (0,1]\}$ has the following properties:
\begin{enumerate}\renewcommand{\labelenumi}{(\roman{enumi})}

  \item  $v_\la\in K(\mathbb{R}^m) \cup \{  \emptyset  \}$ for all $\la\in (0,1]$;

  \item  $v_\la=\bigcap_{\gamma<\lambda}v_\gamma$ for all $\la\in (0,1]$.

\end{enumerate}
In fact,
from property (b) of $\{   u_{q_k}  \}$, we obtain property (\romannumeral1).
Property (\romannumeral2)
can be deduced from
the definition of $\{   v_\al   \}$.

Define    a function
$v: \mathbb{R}^m \to [0,1]$
by
$$v(x)=\left\{
         \begin{array}{ll}
           \bigvee_{x\in v_\lambda}\lambda, & \ \hbox{if } \{ \lambda:  x\in v_\lambda \} \not= \emptyset,
              \\
                 0, & \ \hbox{otherwise.}
         \end{array}
       \right.
$$
Then $v$ is a fuzzy set on $\mathbb{R}^m$.
From properties (\romannumeral1), (\romannumeral2) of $ \{  v_\al   \}$ and Theorem \ref{fuscbchre},
we know that
$$
v \in F_{USCG} (\mathbb{R}^m) \   \hbox{and} \    [v]_\al=v_\al \   \hbox{for each} \ \al \in (0,1].
$$

\textbf{Step 2. } \  Prove that $H_{\rm end} (v_n,   v) \to 0$ as $n\to \infty$.

By Theorem \ref{hpe}, $H_{\rm end} (v_n,   v) \to 0$ as $n\to \infty$
 is equivalent to
\begin{equation}\label{fusbaehc}
  H([v_n]_\alpha, [v]_\alpha) \to  0  \  \hbox{as} \  n\to \infty
\end{equation}
for all $\alpha\in (0,1)\backslash P(v)$.

Take $\alpha\in (0,1)\backslash P(v)$.
If
$ [v]_\al = \emptyset$,
then there is a $q_k < \alpha$
such that
$u_{q_k}=\emptyset$.
Hence,
by \eqref{fusbqpc},
$H([v_n]_{q_k},   u_{q_k}   )=   H([v_n]_{q_k},   \emptyset   ) \to 0$,
and
therefore
$[v_n]_{q_k} \equiv \emptyset$ when $n$ is greater than a certain $N$.
Note
 that
$[v_n]_{\al} \subseteq [v_n]_{q_k} $,
thus
$H([v_n]_{\alpha},   [v]_{\alpha}   )=   H([v_n]_{\alpha},   \emptyset   ) \to 0$.
So \eqref{fusbaehc} holds.

If
$ [v]_\al \not= \emptyset$.
Since $\alpha\in (0,1)\backslash P(v)$,
we know
that
$\al \in (0, \lambda_u)$
and
that, by Theorem \ref{smncbpe},
$H( [v]_\beta,  [v]_\al  ) \to 0$ as $\beta \to \alpha$.
This is equivalent to
$H( u_q,  [v]_\al  ) \to 0$ as $q \to \alpha$.
Thus,
 given $\varepsilon>0$, we can find a $\delta>0$
such that
$H(u_q, v_\al)<\varepsilon$
for all $q\in \mathbb{Q}$ with $|q-\alpha|<\delta$.
So
\begin{equation*}
  H^*([v_n]_\alpha,   v_\al)  \leq    H^*([v_n]_{q_1},   v_\al)  \leq     H^*([v_n]_{q_1},   u_{q_1}) +\varepsilon
\end{equation*}
when $q_1\in \mathbb{Q} \cap (\alpha-\delta,   \alpha)$.
Hence, by eq. \eqref{fusbqpc} and the arbitrariness of $\varepsilon$,
we obtain
\begin{equation}\label{fusblhc}
   H^*([v_n]_\alpha,   v_\al) \to 0 \ (n\to \infty).
\end{equation}
On the other hand,
\begin{equation*}
  H^*(v_\alpha,   [v_n]_\al)  \leq    H^*(v_\al,   [v_n]_{q_2}   )  \leq     H^*( u_{q_2},   [v_n]_{q_2}  ) +\varepsilon
\end{equation*}
when $q_2\in \mathbb{Q} \cap (\alpha,  \alpha + \delta)$.
Hence, by eq. \eqref{fusbqpc} and the arbitrariness of $\varepsilon$,
we obtain
\begin{equation}\label{fusbrhc}
   H^*(  v_\al ,    [v_n]_\alpha ) \to 0 \ (n\to \infty).
\end{equation}
Combined with \eqref{fusblhc} and \eqref{fusbrhc},
we thus
obtain \eqref{fusbaehc}.
 \end{proof}

\begin{tm} \label{fuscbcn}
$U$ is a compact set in $(F_{USCG}(\mathbb{R}^m), H_{\rm end})$
  if and only if $U$ is a closed set in $(F_{USCG}(\mathbb{R}^m), H_{\rm end})$
  and
   $ U(\al) $ is a bounded set in $\mathbb{R}^m $
 when $\al \in (0,1]$.
\end{tm}

\begin{proof}
   \ The desired result follows immediately from Theorem \ref{fuscbrcn}.
\end{proof}

\begin{tm} \label{fusbtoundcn}
Let $U$ be a subset of $ F_{USCG}(\mathbb{R}^m) $.
Then
  $U$ is a totally bounded set in $(F_{USCG}(\mathbb{R}^m), H_{\rm end})$
  if and only if, for each $\al \in (0,1]$,
 $ U(\al) $ is a bounded set in $\mathbb{R}^m $.

\end{tm}

\begin{proof}
 \   \textbf{\emph{Necessity.}} \  We proceed by contradiction. Suppose
that
$U$ is a totally bounded set in $(F_{USCG}(\mathbb{R}^m), H_{\rm end})$.
If
there exists an $\alpha\in (0,1]$ such that
 $ U(\al) $ is not a bounded set in $\mathbb{R}^m $,
then
we can find a sequence $\{u_n\}$ in $U$. It
satisfies that $[u_1]_\alpha \not=  \emptyset$ and
\begin{equation}\label{fsbunc}
  [u_{n+1}]_\al \nsubseteq V_n,
\end{equation}
where
$V_n= \{x\in \mathbb{R}^m :     d(x,   \bigcup_{i=1}^n  [ u_i ]_\al ) \leq n       \}    $.
Assume
that
$\{w_{1}, w_{2},   \ldots,    w_{k}\}$
is
a $\alpha/2$-net of $U$.
Clearly,
$\bigcup_{i=1}^k   [w_{i}]_{\alpha/3}$
is a bounded set in
$  \mathbb{R}^m$.
Hence
there exists an $N \in  \mathbb{N}$ such that
$$  [w_{i}]_{\alpha/3} \subseteq   V_N,   \ \hbox{for }  i=1,2,\ldots k,$$
and
therefore, by \eqref{fsbunc},
$$H( [ u_{N+3}]_\alpha ,      \bigcup_{i=1}^{k}  [w_{i}]_{\alpha/3}       ) >  1. $$
Thus
$$
H (   {\rm end}\,   u_{N+3} ,    {\rm end}\,  w_{i}   ) \geq    \frac{2}{3}  \alpha
\ \hbox{for }  i=1,2,\ldots k.
$$
This contradicts
the assumption
that
$\{w_{1}, w_{2},   \ldots,    w_{k}\}$
is
a $\alpha/2$-net of $U$.

\textbf{\emph{Sufficiency.}} \ Suppose that  $ U(\al) $ is a bounded set in $\mathbb{R}^m $
for
each $\al\in (0,1]$.
Then,
by Theorem \ref{fuscbrcn},
we know
that
$U$ is relatively compact,
and
thus
$U$ is totally bounded.
\end{proof}

\begin{tl} \label{fusbtdrceqn}
  Let $U$ be a set in $(F_{USCG}(\mathbb{R}^m), H_{\rm end})$. The $U$ is totally bounded
  if and only if
  $U$ is relatively compact.
\end{tl}

\begin{proof}
 The desired result follows immediately
from
Theorems \ref{fuscbrcn} and \ref{fusbtoundcn}.
\end{proof}

\begin{tm} \label{fusbcspe}
  $(F_{USCG}(\mathbb{R}^m), H_{\rm end})$ is a complete space.
\end{tm}

\begin{proof}
   The desired result follows
from Corollary \ref{fusbtdrceqn}.
 \end{proof}

\begin{re}\label{trm}{\rm The relationships of endograph metric and $\Gamma$-convergence given in Section \ref{rghc}
have deep connections with
the characterizations of relative compactness and total boundedness obtained
in this section.

In this paper,
we first give
the relationship
of
endograph metric and $\Gamma$-convergence
(Theorem \ref{hgce}), then present
 level characterizations of endograph metric $H_{\rm end}$
of
fuzzy sets in $F_{USCG} ( \mathbb{R}^m )$ (Theorem \ref{hpe}),
and
then obtain
characterizations of relative compactness, compactness and total boundedness
of $(F_{USCG}(\mathbb{R}^m), H_{\rm end})$
(Theorems \ref{fuscbrcn}, \ref{fuscbcn} and \ref{fusbtoundcn}).
The latter results are proved by using
the former
results.

Here, we mention that,
conversely,
by
 using characterization of relatively compact sets in $(F_{USCG}(\mathbb{R}^m), H_{\rm end})$
(Theorem \ref{fuscbrcn}), we can also deduce
the relationship of
 the
endograph metric and the $\Gamma$-convergence
(Theorem \ref{hgce}).

For the convenience of reading, we write
Theorem \ref{hgce}
in the following.

\begin{description}
  \item[] \hspace{6mm}
 Suppose that $u$, $u_n$,
$n=1,2,\ldots$, are  fuzzy sets in $F_{USCG} (\mathbb{R}^m)$.
  Then $H_{\rm end}(u_n, u) \to 0$ as $n\to \infty$
is equivalent to
\\
  (\romannumeral1)  $u_n\str{\Gamma}{\longrightarrow}u$, and
\\
 (\romannumeral2)  given $\alpha\in (0,1]$,
$ \bigcup_{n=1}^{+\infty}  [u_n]_\al $ is a bounded set in $\mathbb{R}^m $.
\end{description}

\begin{proof}
\textbf{Sufficiency}. If $H_{\rm end}(u_n, u) \not\to 0$, then there is an $\varepsilon>0$ and a subsequence $\{u_{n_i}\}$ such that
\begin{equation}\label{evu}
H_{\rm end}(u_{n_i}, u) \geq \varepsilon.
\end{equation}
From Theorem  \ref{fuscbrcn}, $\{u_n\}$ is a relatively compact set in $(F_{USCG}(\mathbb{R}^m), H_{\rm end})$.
Thus
$u$
is
an
accumulation point of $\{u_{n_i}\}$ with respect to $H_{\rm end}$, which contradicts \eqref{evu}.

{\textbf{Necessity}.} (\romannumeral1) follows from Theorem \ref{uschpm}. (\romannumeral2) follows from
the fact $\{u_n\}$ is relatively compact and Theorem \ref{fuscbrcn}.
\end{proof}

}
\end{re}

\section{Properties of $(F_{USCB} (\mathbb{R}^m), H_{\rm end})$     \label{uscbpy}}

  In this section, we point out that $(F_{USCG} (\mathbb{R}^m), H_{\rm end})$
  is the completion
  of
  $(F_{USCB} (\mathbb{R}^m), H_{\rm end})$,
 and
    give characterizations of totally bounded sets, relatively compact sets and compact sets in
$(F_{USCB}(\mathbb{R}^m), H_{\rm end})$.

Ma \cite{ma}
use $u^{(\alpha)}$ to denote the fuzzy set $u^{(\alpha)}$ derived by $u\in F(\mathbb{R}^m)$ which is defined as follows:
\[
u^{(\alpha)}(x)=
\left\{
  \begin{array}{ll}
    u(x), & \hbox{if} \  u(x) \geq\alpha,
\\
    0, & \hbox{if} \   u(x)<\alpha.
  \end{array}
\right.
\]

\begin{lm} \label{csn}
  Let $u\in F_{USCG}(\mathbb{R}^m)$. Then, for each $\al\in (0,1]$,
   $H( [u^{(1/n)}]_\al, [u]_\al) \to 0$
   as
   $n\to \infty$. Thus $H_{\rm end} (u^{(1/n)}, u) \to 0$ as $n\to \infty$.
\end{lm}

\begin{proof}
 \ Note that, for any $\beta\in [0, 1]$,
\[
[u^{(\alpha)} ]_\beta=
\left\{
  \begin{array}{ll}
   [u]_\beta, & \hbox{if} \ \beta \geq\alpha,
\\
   \mbox{} [u]_\al, & \hbox{if} \   \beta<\alpha.
  \end{array}
\right.
\]
Hence we can see that, given $\al\in (0,1]$,
   $H( [u^{(1/n)}]_\al, [u]_\al) \to 0$
   as
   $n\to \infty$.
    $H( [u^{(1/n)}]_\al, [u]_\al) \to 0$  for each $\al \in (0,1]$.
Hence by Theorem \ref{hpe},
$H_{\rm end} (u^{(1/n)}, u) \to 0$ as $n\to \infty$.
\end{proof}

\begin{tm} \label{bdg}
 $ F_{USCB}(\mathbb{R}^m)$ is dense in  $( F_{USCG}(\mathbb{R}^m), H_{\rm end})$.
\end{tm}

\begin{proof}
  \ Take $u\in F_{USCG}(\mathbb{R}^m)$. Then $\{u^{(1/n)}, n=1,2,\ldots\}$ is a sequence in   $F_{USCB}(\mathbb{R}^m)$.
By
Lemma \ref{csn},
$H_{\rm end} (u^{(1/n)}, u) \to 0$ as $n\to \infty$.
So
 $F_{USCB}(\mathbb{R}^m)$ is dense in  $(F_{USCG}(\mathbb{R}^m), H_{\rm end})$.
 \end{proof}

\begin{tm} \label{gcde}
 $ (F_{USCG}(\mathbb{R}^m), H_{\rm end})$ is the completion of $( F_{USCB}(\mathbb{R}^m), H_{\rm end})$.
\end{tm}

\begin{proof}
\ From Theorems \ref{fusbcspe} and \ref{bdg}, we can obtain the desired result.
\end{proof}

\begin{tm} \label{fusbtdn}
Let $U$ be a subset of $ F_{USCB}(\mathbb{R}^m) $.
Then
  $U$ is a totally bounded set in $(F_{USCB}(\mathbb{R}^m), H_{\rm end})$
  if and only if, for each $\al \in (0,1]$,
 $ U(\al) $ is a bounded set in $\mathbb{R}^m $.

\end{tm}

\begin{proof}
   \ The desired result follows immediately
from
Theorem \ref{fusbtoundcn}
and
the fact
that
$ F_{USCB}(\mathbb{R}^m) \subset  F_{USCG}(\mathbb{R}^m)  $.
\end{proof}

Given $U$ in  $F_{USCG}(\mathbb{R}^m)$,
the symbol
\bm{$\overline{U}$} is used to denote the closure of $U$ in  $(F_{USCG}(\mathbb{R}^m), H_{\rm end})$.

\begin{tm} \label{brec}
  Let $U$ be a subset of $ F_{USCB}(\mathbb{R}^m) $.
Then
  $U$ is a relatively compact set in $(F_{USCB}(\mathbb{R}^m), H_{\rm end})$
  if and only if
  \\
 (\romannumeral1)  For each $\al \in (0,1]$,
 $ U(\al) $ is a bounded set in $\mathbb{R}^m $, and
 \\
  (\romannumeral2)  $\overline{U} \subset F_{USCB}(\mathbb{R}^m)$.
\end{tm}

\begin{proof} \ The desired result follows immediately from Theorems \ref{fuscbrcn}. \end{proof}

\begin{tm} \label{bcn}
  Let $U$ be a subset of $ F_{USCB}(\mathbb{R}^m) $.
Then
  $U$ is a compact set in $(F_{USCB}(\mathbb{R}^m), H_{\rm end})$
  if and only if
  \\
 (\romannumeral1)  For each $\al \in (0,1]$,
 $ U(\al) $ is a bounded set in $\mathbb{R}^m $, and
 \\
  (\romannumeral2)  $\overline{U} = U$.
\end{tm}

\begin{proof} \ The desired result follows immediately from Theorem \ref{fuscbcn}. \end{proof}

Now, we consider a characterization of relatively compact set $U$ in
$(F_{USCB}   (  \mathbb{R}^m ), H_{\rm end})$
which
do
not
involve the closure of $U$ in $(F_{USCG}(\mathbb{R}^m), H_{\rm end})$.

Let $B_r:=\{ x\in \mathbb{R}^m :     \|x\| \leq r    \}$
and
$\widehat{B_r}$ be the characteristic function
of $B_r$, where $r$ is a positive real number.
Given $u\in F_{USCG} (\mathbb{R}^m)$.
Then
$u \vee \widehat{B_r}  \in F_{USCG} (\mathbb{R}^m)$.
 Define
$$    |  u |_r:  = H_{\rm end}   (  u \vee \widehat{B_r}   ,       \widehat{B_r}    ).
$$
It
 can be checked that, for $u\in F_{B} (\mathbb{R}^m)$, $ |  u |_r=0$ if and only if
$[u]_0 \subseteq  B_r  $.
Note
that
$$  H_{\rm end} (  u \vee \widehat{B_r},      v \vee \widehat{B_r}   ) \leq  H_{\rm end} ( u,   v ),    $$
it thus holds
that
\begin{equation}\label{nep}
  |   |u|_r  - |v|_r   | \leq      H_{\rm end} ( u,   v ).
\end{equation}

\begin{tm} \label{fbsrchra} Let
  $U \subset F_{USCB}   (  \mathbb{R}^m )$.
Then $U$ is relatively   compact in $ (F_{USCB}   (  \mathbb{R}^m ),    H_{\rm end})$
 if and only if
 \\
(\romannumeral1)
 For each $\al \in (0,1]$,
 $ U(\al) $ is a bounded set in $\mathbb{R}^m $, and
 \\
(\romannumeral2$'$)
Given $\{u_n: n=1,2,\ldots\} \subset U$, there exists a $r>0$ and a subsequence $\{v_n\}$ of $\{u_n\}$ such
that
$\lim_{n\to \infty}    |v_n|_r=0$.

\end{tm}

\begin{proof}
 \ Suppose
that
 $U$ is a relatively compact set which does not satisfy the condition (\romannumeral2$'$).
Take
 $r=1$.
Then
there exists $\varepsilon_1 >0$ and a subsequence $\{u_{n}^{(1)}: n=1,2,\ldots\}$
of
 $\{u_n: n=1,2,\ldots\}$
such
that
$ |u_n^{(1)}|_1 > \varepsilon_1$ for all $n=1,2,\ldots$.
If
  $\{u_n^{(1)}\}, \ldots, \{u_n^{(k)}\}$ and positive numbers $\varepsilon_1, \ldots,  \varepsilon_k$ have been found,
then
we
can
find
a subsequence $\{    u_n^{  (k+1)  }    \}$
of
$\{   u_n^ { (k) }   \}$
and
$\varepsilon_{k+1} >0$
such that
 $ |u_n^{(k+1)}|_{k+1} > \varepsilon_{k+1}$ for all $n=1,2,\ldots$.
Set $v_n=   u_n^{  (n) }    $ for $n=1,2,\ldots$.
Then $\{v_n\}$
is
a subsequence of
$\{u_n\}$
and
\begin{equation}\label{vne}
 \liminf_{n\to \infty}    |v_n|_k    \geq    \varepsilon_k
\end{equation}
for   $k=1,2,\ldots$.
Suppose that $v \in  F_{USCG}   (  \mathbb{R}^m )$ is an accumulation point of $\{v_n\}$.
Then
from \eqref{nep} and \eqref{vne},
we know
that
\begin{equation*}
   |v|_k    \geq   \varepsilon_k > 0
\end{equation*}
for all $k=1,2,\ldots$. Thus $v \notin  F_{USCB}   (  \mathbb{R}^m ) $.
This contradicts
the fact
that
 $U$ is a relatively compact set in $(  F_{USCB}   (  \mathbb{R}^m ), H_{\rm end}  )$.

Suppose that $U \subset  F_{USCB}   (  \mathbb{R}^m )$ satisfies the condition (\romannumeral2$'$).
Given a sequence
$\{u_n\}$ in
$U$ with $\lim_{n\to \infty} u_n = u \in F_{USCG}   (  \mathbb{R}^m )$.
Then, from
\eqref{nep},
there exists a $r>0$ such
that
$\lim_{n\to \infty}    |u_n|_r=|u|_r=0$.
Hence $[u]_0 \subseteq  B_r   $,
i.e.
$u \in F_{USCB}   (  \mathbb{R}^m )$.
So
 $\overline{U} \subset F_{USCB}(\mathbb{R}^m)$.
Thus, by Theorem \ref{brec}, if $U$ meets the conditions (\romannumeral1) and (\romannumeral2$'$),
then
$U$
is a relatively compact set in $(  F_{USCB}   (  \mathbb{R}^m ), H_{\rm end})$.
\end{proof}

\begin{tm}
  Let
  $U \subset F_{USCB}   (  \mathbb{R}^m )$.
  Then the following statements are
  equivalent.
  \\
  (\romannumeral1) \    $U$ is   compact in $ (F_{USCB}   (  \mathbb{R}^m ),    H_{\rm end})$.
  \\
    (\romannumeral2) \    $U$
 satisfies the
 conditions (\romannumeral1) and (\romannumeral2$'$) in Theorem \ref{fbsrchra} and $U$ is   closed in $ (F_{USCB}   (  \mathbb{R}^m ),    H_{\rm end})$.
\end{tm}

\begin{proof} \ The desired conclusions follow immediately from Theorem    \ref{fbsrchra}. \end{proof}

\section{Properties and relationship of endograph metric and $\Gamma$-convergence on subsets of $F_{USCG} (\mathbb{R}^m)$ \label{pfuscg}}

\subsection{Relationship and level characterizations of endograph metric and $\Gamma$-convergence on $F_{USCGCON} (\mathbb{R}^m)$} \label{sbp}

In this subsection,
 we present level characterizations of
the $\Gamma$-convergence on $F_{USCGCON} (\mathbb{R}^m)$ based on preceding results.
By using
these level characterizations and the relationships of endograph metric and $\Gamma$-convergence obtained in Section \ref{rghc},
we find that
the endograph metric convergence
 and the $\Gamma$-convergence
\textbf{coincide}
on
 $F_{USCGCON} (\mathbb{R}^m) \backslash \widehat{\emptyset}$.
As a consequence, we give
the forms of the level characterizations of the endograph metric and the $\Gamma$-convergence
when
restricted to    $F_{USCGCON} (\mathbb{R}^m) \backslash \widehat{\emptyset}$ and common fuzzy sets, respectively.
It is found
that
the level characterizations
of
 the endograph metric convergence
 and the $\Gamma$-convergence given in Section 6
are
\emph{in some sense} the same on
 $F_{USCGCON} (\mathbb{R}^m) \backslash \widehat{\emptyset}$
and
are \emph{precisely} the same
on any common fuzzy sets mentioned in this paper.
 At last, based on the conclusions of this section,
we revisit
the results in previous works.

\begin{tm} \label{ltgcfbn}
 Suppose that $u_n$,
$n=1,2,\ldots$, are  fuzzy sets in $F_{USCGCON} (\mathbb{R}^m)$,
and that
$u \not= \widehat{\emptyset} $ is a fuzzy set in $F_{USCG} (\mathbb{R}^m)$.
Then
 the
following statements are equivalent.
\\
(\romannumeral1)\ $u_n\str{\Gamma}{\longrightarrow}u$.
\\
(\romannumeral2)\  $\lim_{n\to \infty} [u_n]_\alpha (K)= [u]_\al$ for all
$\al\in (0,1)\backslash P(u)$.
\\
(\romannumeral3)\  $\lim_{n\to \infty} [u_n]_\al (K)  = [u]_\al = \emptyset$ for each $\al\in (\lambda_u, 1]$,
and
$H([u_n]_\al,  [u]_\al) \to 0$
for all
$\al\in (0,\lambda_u)\backslash P(u)$. (If $\lambda_u = 1$,
then
 $(\lambda_u, 1]=(1, 1] = \emptyset$.)
\\
(\romannumeral4)\ $\lim_{n\to \infty} [u_n]_\alpha (K)= [u]_\al$ holds a.e. on $\al\in (0,1)$.
\\
(\romannumeral5)\ $\lim_{n\to \infty} [u_n]_\al (K)= [u]_\al = \emptyset $ for each $\al\in (\lambda_u, 1]$,
and
$H([u_n]_\al,  [u]_\al) \to 0$ holds a.e. on $\alpha \in (0,\lambda_u)$.
\\
(\romannumeral6)\  $\lim_{n\to \infty} [u_n]_\alpha (K)= [u]_\al$ holds when
$\al\in P$, where $P$ is a dense subset of $(0,1) \backslash P(u)$.
\\
(\romannumeral7)\ $\lim_{n\to \infty} [u_n]_\al (K)= [u]_\al = \emptyset $ for each $\al\in (\lambda_u, 1]$,
 and
$H([u_n]_\al,  [u]_\al) \to 0$ holds when $\alpha \in P  \cap  (0,\lambda_u)$, where $P$ is a dense subset of $(0,1) \backslash P(u)$.
  \\
  (\romannumeral8)\  $\lim_{n\to \infty} [u_n]_\alpha (K)= [u]_\al$ holds when $\al\in P$, where $P$ is a countable dense subset of $(0,1) \backslash P(u)$.
\\
(\romannumeral9)\ $\lim_{n\to \infty} [u_n]_\al (K)= [u]_\al = \emptyset $ for each $\al\in (\lambda_u, 1]$,
 and
$H([u_n]_\al,  [u]_\al) \to 0$ holds when $\alpha \in P  \cap  (0,\lambda_u)$, where $P$ is a countable dense subset of $(0,1) \backslash P(u)$.
\end{tm}

\begin{proof}
   \ Let $\al\in (0,1]$. Note that
$[u_n]_\al$, $n=1,2,\ldots$,
are compact and connected sets in $\mathbb{R}^m$, and that
$[u]_\al$
is a compact set in $\mathbb{R}^m$.
If
 $
[u]_\al \not= \emptyset$,
then,
by Propositions \ref{hms} and \ref{klhe},
we know that
$H([u_n]_\al,  [u]_\al) \to 0$
is equivalent to
 $\lim_{n\to \infty} [u_n]_\alpha (K)= [u]_\al$.
So the desired results follow immediately from
Theorem \ref{ltgc}.
\end{proof}

\begin{lm}\label{bconedu}
   Suppose that $u_n$,
$n=1,2,\ldots$, are  fuzzy sets in $F_{USCGCON} (\mathbb{R}^m)$,
and that
$u \not= \widehat{\emptyset}$ is a fuzzy set in $F_{USCG} (\mathbb{R}^m)$.
   If
   $u_n\str{\Gamma}{\longrightarrow}u$, then, for each $\al\in (0,1]$,
$ \bigcup_{n=1}^{+\infty}  [u_n]_\al $ is a bounded set in $\mathbb{R}^m $.
\end{lm}

\begin{proof}
\ Suppose that $\al\in (0,1]$. Note that $u \not= \widehat{\emptyset}$.
By Theorem \ref{ltgcfbn},
there exists a $\beta < \al$
such that
$H([u_n]_\beta,  [u]_\beta) \to 0$.
Thus
$ \bigcup_{n=1}^{+\infty}  [u_n]_\beta $ is a bounded set in $\mathbb{R}^m $,
and
so
$ \bigcup_{n=1}^{+\infty}  [u_n]_\al $ is also a bounded set in $\mathbb{R}^m $.
\end{proof}

\begin{tm} \label{gchefbcon}
  Suppose that $u_n$,
$n=1,2,\ldots$, are  fuzzy sets in $F_{USCGCON} (\mathbb{R}^m)$,
and
that
$u \not= \widehat{\emptyset}$ is a fuzzy set in $F_{USCG} (\mathbb{R}^m)$.
   Then
   $u_n\str{\Gamma}{\longrightarrow}u$
  is equivalent to
   $H_{\rm end} ( u_n, u) \to 0$ as $n\to \infty$.
\end{tm}

\begin{proof}  The desired conclusion follows from Corollary \ref{gchrbsn} and Lemma \ref{bconedu}.
\end{proof}

\begin{re}{\rm
 Theorem \ref{gchefbcon}
exhibits
a fascinating fact:
the $\Gamma$-convergence can be induced by the endograph metric
on
$F_{USCGCON} (\mathbb{R}^m) \backslash \widehat{\emptyset}$. Namely,
 the Hausdorff metric  \textbf{metrizes} the Fell topology on the set of endographs of fuzzy sets in $F_{USCGCON} (\mathbb{R}^m) \backslash \widehat{\emptyset}$.
 }
\end{re}

\begin{tm} \label{lthfbn}
 Suppose that $u_n$,
$n=1,2,\ldots$, are  fuzzy sets in $F_{USCGCON} (\mathbb{R}^m)$,
and that
$u \not= \widehat{\emptyset}$ is a fuzzy set in $F_{USCG} (\mathbb{R}^m)$.
Then the
following statements are equivalent.
\\
(\romannumeral1)\ $ H_{\rm end} ( u_n, u)  \to 0 $ as $n\to \infty$.
\\
(\romannumeral2)\ $H([u_n]_\al,  [u]_\al) \to 0$ holds a.e. on $\al\in (0,1)$.
\\
(\romannumeral3)\  $H([u_n]_\al,  [u]_\al) \to 0$ holds when
$\al\in (0,1)\backslash P(u)$.
\\
 (\romannumeral4)\   $H([u_n]_\al,  [u]_\al) \to 0$  holds when
$\al\in P$, where $P$ is a dense subset of $(0,1) \backslash P(u)$.
  \\
  (\romannumeral5)\  $H([u_n]_\al,  [u]_\al) \to 0$  holds when $\al\in P$, where $P$ is a countable dense subset of $(0,1) \backslash P(u)$.
\\
(\romannumeral6)\ $u_n\str{\Gamma}{\longrightarrow}u$.
\\
(\romannumeral7)\  $\lim_{n\to \infty} [u_n]_\alpha (K)= [u]_\al$ for all
$\al\in (0,1)\backslash P(u)$.
\\
(\romannumeral8)\  $\lim_{n\to \infty} [u_n]_\al (K)  = [u]_\al = \emptyset$ for each $\al\in (\lambda_u, 1]$,
and
$H([u_n]_\al,  [u]_\al) \to 0$
for all
$\al\in (0,\lambda_u)\backslash P(u)$.
\\
(\romannumeral9)\ $\lim_{n\to \infty} [u_n]_\alpha (K)= [u]_\al$ holds a.e. on $\al\in (0,1)$.
\\
(\romannumeral10)\ $\lim_{n\to \infty} [u_n]_\al (K)= [u]_\al = \emptyset $ for each $\al\in (\lambda_u, 1]$,
and
$H([u_n]_\al,  [u]_\al) \to 0$ holds a.e. on $\alpha \in (0,\lambda_u)$.
\\
(\romannumeral11)\  $\lim_{n\to \infty} [u_n]_\alpha (K)= [u]_\al$ holds when
$\al\in P$, where $P$ is a dense subset of $(0,1) \backslash P(u)$.
\\
(\romannumeral12)\ $\lim_{n\to \infty} [u_n]_\al (K)= [u]_\al = \emptyset $ for each $\al\in (\lambda_u, 1]$,
 and
$H([u_n]_\al,  [u]_\al) \to 0$ holds when $\alpha \in P  \cap  (0,\lambda_u)$, where $P$ is a dense subset of $(0,1) \backslash P(u)$.
  \\
  (\romannumeral13)\  $\lim_{n\to \infty} [u_n]_\alpha (K)= [u]_\al$ holds when $\al\in P$, where $P$ is a countable dense subset of $(0,1) \backslash P(u)$.
\\
(\romannumeral14)\ $\lim_{n\to \infty} [u_n]_\al (K)= [u]_\al = \emptyset $ for each $\al\in (\lambda_u, 1]$,
 and
$H([u_n]_\al,  [u]_\al) \to 0$ holds when $\alpha \in P  \cap  (0,\lambda_u)$, where $P$ is a countable dense subset of $(0,1) \backslash P(u)$.
\end{tm}

\begin{proof} \    The desired results follow from Theorems \ref{hpe}, \ref{ltgcfbn} and \ref{gchefbcon}. \end{proof}

\begin{re}
 {\rm
 From Theorem \ref{lthfbn}, we can see that
 the level characterizations of the endograph metric and the $\Gamma$-convergence
 given in Theorems \ref{ltgc} and \ref{hpe}
 are in some sense the
 same on $F_{USCGCON} (\mathbb{R}^m) \backslash \widehat{\emptyset}$.
}
\end{re}

\begin{tm} \label{cgme}
 Suppose that $ u$, $u_n$,
$n=1,2,\ldots$, are  fuzzy sets in $\widetilde{S}^{m}_{nc}$.
Then the
following statements are equivalent.
\\
(\romannumeral1) $H_{\rm end} (u_n, u) \to 0$ as $n\to \infty$.
\\
(\romannumeral2)\ $H([u_n]_\al,  [u]_\al) \to 0$ holds a.e. on $\al\in (0,1)$.
\\
(\romannumeral3)\  $H([u_n]_\al,  [u]_\al) \to 0$ holds when
$\al\in (0,1)\backslash P(u)$.
\\
 (\romannumeral4)\   $H([u_n]_\al,  [u]_\al) \to 0$  holds when
$\al\in P$, where $P$ is a dense subset of $(0,1) \backslash P(u)$.
  \\
  (\romannumeral5)\  $H([u_n]_\al,  [u]_\al) \to 0$  holds when $\al\in P$, where $P$ is a countable dense subset of $(0,1) \backslash P(u)$.
\\
(\romannumeral6)\ $u_n\str{\Gamma}{\longrightarrow}u$.
\\
(\romannumeral7)\  $\lim_{n\to \infty} [u_n]_\alpha (K)= [u]_\al$ for all
$\al\in (0,1)\backslash P(u)$.
\\
(\romannumeral8)\ $\lim_{n\to \infty} [u_n]_\alpha (K)= [u]_\al$ holds a.e. on $\al\in (0,1)$.
\\
(\romannumeral9)\  $\lim_{n\to \infty} [u_n]_\alpha (K)= [u]_\al$ holds when
$\al\in P$, where $P$ is a dense subset of $(0,1) \backslash P(u)$.
 \\
  (\romannumeral10)\  $\lim_{n\to \infty} [u_n]_\alpha (K)= [u]_\al$ holds when $\al\in P$, where $P$ is a countable dense subset of $(0,1) \backslash P(u)$.

\end{tm}

\begin{proof} \ Note that $\widetilde{S}^{m}_{nc} \subset  F_{USCGCON} (\mathbb{R}^m)$ and that for each
$u\in \widetilde{S}^{m}_{nc}$, $[u]_1 \not= \emptyset$ and hence $u \not= \widehat{\emptyset} $. So
the
desired conclusion follows immediately
from
Theorem \ref{lthfbn}.
\end{proof}

\begin{re}
{\rm
From Theorem \ref{cgme}, we know that the $\Gamma$-convergence
and
 the endograph metric convergence
are
equivalent
on common fuzzy sets such as
fuzzy numbers (compact and noncompact),
 fuzzy  star-shaped  numbers (compact and noncompact),
and
 general fuzzy   star-shaped   numbers (compact and noncompact);
 that is to say,
 the Fell
topology is precisely the Hausdorff metric topology on the set of endographs of
these common fuzzy sets.
  }
\end{re}

\begin{re}
{\rm
Suppose that $u, u_n \in\widetilde{S}^{m}_{nc}$, $n=1,2,\ldots$, and that $\al\in (0,1]$.
 From Theorem \ref{cgme}, $ H_{\rm end} ( u_n, u)  \to 0 $ is equivalent to $u_n\str{\Gamma}{\longrightarrow}u$.
From Propositions \ref{hms} and \ref{klhe}, $H([u_n]_\al,  [u]_\al) \to 0$ is equivalent to $\lim_{n\to \infty} [u_n]_\alpha (K)= [u]_\al$.
So
 the level characterizations of the endograph metric and the $\Gamma$-convergence
 given in Theorems \ref{ltgc} and \ref{hpe}
 are precisely the
 same on $\widetilde{S}^{m}_{nc}$.

In fact, it is easy to check
that
 the level characterizations of the endograph metric and the $\Gamma$-convergence
are
precisely the
 same on $F_{USCGCON}^N (\mathbb{R}^m):=\{ u\in F_{USCGCON} : [u]_1\not=\emptyset    \}$.
}
 \end{re}

\begin{re}\label{ems}
{\rm
 These level characterizations help us to see the previous results more clearly.
 Note that for $u$, $u_n$, $n=1,2,\ldots$, in $F_{USC}(\mathbb{R}^m)$,
\begin{gather}
  \lim_{n\to \infty}{\rm send}\, u_n (K)= {\rm send}\, u
\ \hbox{iff} \
u_n\str{\Gamma}{\longrightarrow}u
\ \hbox{and}
\lim_{n\to \infty}  [u_n]_0 (K)=  [u]_0,   \nonumber
 \\
H ({\rm send}\, u_n,  {\rm send}\, u ) \to 0  \ \hbox{iff} \   H_{\rm end} (u_n, u) \to 0 \ \hbox{and} \ H([u_n]_0, [u]_0) \to 0.  \label{se0r}
\end{gather}
So
 the level decomposition properties of the $\Gamma$-convergence
(the three equivalent statements (\romannumeral6), (\romannumeral7) and (\romannumeral8) in
Theorem \ref{cgme})
 can imply some previous results such as Propositions 8, 12, 16 and Theorems 3, 14 in \cite{trutschnig}.

By using \eqref{se0r},
it can be checked that
Theorem \ref{hdp}
can
deduce Propositions 6 in \cite{trutschnig}, and that    Theorems \ref{hdp} and \ref{dpf} can imply Theorem 11 in \cite{trutschnig}.
 }
\end{re}

\begin{re}\label{rpe}
{\rm
By using Propositions \ref{hms} and \ref{klhe}, we can relook some previous conclusions.
For example, Lemma 2.5 in \cite{huang7} is a known conclusion;
Propositions 8 and 12 in \cite{trutschnig} are the same;
 the second conditions of
 Propositions 8 and 16 in \cite{trutschnig}
 are
 the same.

 }
 \end{re}

\subsection{The relationships among subspaces of $( F_{USCG} (\mathbb{R}^m),   H_{\rm end} )$ \label{resucon}}

We discuss the relationships among various subspaces of $( F_{USCG} (\mathbb{R}^m),   H_{\rm end} )$.
One of the results
is
that
the fuzzy set spaces of noncompact type   are exactly the completions of their compact counterparts under the endograph metric.
All the results
obtained
in this subsection
are
summarized in Figure \ref{uscbren}.

\begin{tm} \label{ncduscb}
  $\widetilde{S}^m_{nc}$ is a closed set in $(  F_{USCG} (\mathbb{R}^m), H_{\rm end}   )$.
\end{tm}

\begin{proof} \ We only need to show that each limit point of $\widetilde{S}^m_{nc}$ belongs to
itself.
Let   $\{u_n\}$ be a sequence in    $\widetilde{S}^m_{nc}$ with $\lim u_n =u \in F_{USCG} (\mathbb{R}^m)$.
Then by Theorem \ref{cgme},
$H([u_n]_\alpha, [u]_\alpha) \str{   \mbox{a.e.}   }{\to}   0 \ (  [0,1]    )$.
We affirm
that
$[u]_\al \in K_S(\mathbb{R}^m)$ for each $\al\in (0,1]$.
In fact,
take
$\al \in (0,1]$. If $H([u_n]_\alpha, [u]_\alpha) \to   0 $,
then
by Theorem \ref{ksc},
$[u]_\al \in K_S(\mathbb{R}^m)$.
 If $H([u_n]_\alpha, [u]_\alpha) \not\to   0 $,
then
there exists a sequence
$\beta_n  \to \al-$
such that
$ [u]_{\beta_n} \in  K_S(\mathbb{R}^m)$.
Note that
$[u]_\al= \bigcap_{n} [u]_{\beta_n} $,
and
this yields
$ [u]_\alpha \in  K_S(\mathbb{R}^m)$.
From Theorems \ref{gsmre} and \ref{fuscbchre},
we
thus
know that
$u \in \widetilde{S}^m_{nc}$.
\end{proof}

\begin{tm} \label{dinc}
 $\widetilde{S}^m$ is dense in  $(\widetilde{S}^m_{nc}, H_{\rm end})$.
\end{tm}

\begin{proof} \ Take $u\in \widetilde{S}^m_{nc}$. Then $\{u^{(1/n)}, n=1,2,\ldots\}$ is a sequence in   $\widetilde{S}^m$.
By
Lemma \ref{csn},
$H_{\rm end} (u^{(1/n)}, u) \to 0$ as $n\to \infty$.
So
 $\widetilde{S}^m$ is dense in  $(\widetilde{S}^m_{nc}, H_{\rm end})$.
\end{proof}

\begin{tm} \label{wsnclp}
   $(\widetilde{S}^m_{nc}, H_{\rm end})$ is the completion of  $(\widetilde{S}^m, H_{\rm end})$.
\end{tm}

\begin{proof}
\ The desired result follows immediately from Theorems \ref{fusbcspe}, \ref{ncduscb} and \ref{dinc}.
 \end{proof}

\begin{tm} \label{smcg}
   $S^m_{nc}$ is a closed set in $( \widetilde{S}^m_{nc}, H_{\rm end}   )$.
\end{tm}

\begin{proof} \  To show the closedness of   $S^m_{nc}$ in $( \widetilde{S}^m_{nc}, H_{\rm end}   )$.
Let $\{u_n , n=1,2,\ldots\}$ be a sequence in $S^m_{nc}$
  and
 $\lim_{n\to \infty} u_n=u \in \widetilde{S}^m_{nc}$.
We only need to
prove
  $u \in S^m_{nc}$.
From
Remark \ref{nce}, this
 is equivalent
 to
 prove $\mbox{ker}\,  u \not=\emptyset$.

From
 Theorem \ref{cgme},
 $H_{\rm end}(u_n, u) \to 0$
is equivalent to
\begin{equation*} \label{hue}
  H([u_n]_\al,  [u]_\al) \to 0 \ \hbox{a.e. on } \ [0,1].
\end{equation*}
By
proceeding according to the proof of $\mbox{ker}\,  u \not=\emptyset$ in Theorem 6.3 of \cite{huang9},
we
can
show
 $\mbox{ker}\,  u \not=\emptyset$.
\end{proof}

\begin{tm} \label{dis}
 $S^m$ is dense in  $(S^m_{nc}, H_{\rm end})$.
\end{tm}

\begin{proof} \  Take $u\in S^m_{nc}$. Then $\{u^{(1/n)}, n=1,2,\ldots\}$ is a sequence in   $S^m$.
By
Lemma \ref{csn},
$H_{\rm end} (u^{(1/n)}, u) \to 0$ as $n\to \infty$.
So
 $S^m$ is dense in  $(S^m_{nc}, H_{\rm end})$. \end{proof}

\begin{tm} \label{smnc}
  $( S^m_{nc},    H_{\rm end}  )$ is the completion of $( S^m,   H_{\rm end} )$.
\end{tm}

\begin{proof}
\ The desired result follows immediately from Theorems \ref{wsnclp}, \ref{smcg}, \ref{dis}.
\end{proof}

\begin{tm}
  $S^m$ is  a closed set in $( \widetilde{S}^m,    H_{\rm end}   )$.
\end{tm}

\begin{proof} \ The desired conclusion follows immediately from Theorem \ref{smcg}. \end{proof}

\begin{tm} \label{emcg}
 $E^m_{nc} $ is a closed set in $(S^m_{nc},  H_{\rm end}  )$.
\end{tm}

\begin{proof} \   Let $u_n $, $n=1,2,\ldots$, be fuzzy sets in $E^m_{nc}$
  and
 $\lim_{n\to \infty} u_n = u$ be a fuzzy set in $S^m_{nc}$.
To show the desired result, we only need to prove
  $u \in E^m_{nc}$.
  This is equivalent to show $[u]_\al \in K_C(\mathbb{R}^m)$ for any $\al\in (0,1]$.
From
 Proposition \ref{kcs},
 we know $(K_C(\mathbb{R}^m), H)$
is a
complete space.
The remainder of the proof is similar to
the
proof
of
Theorem \ref{smcg}.
\end{proof}

\begin{tm}
  $E^m$ is a closed set in $(S^m, H_{\rm end})$.
\end{tm}

\begin{proof} \ The desired result follows immediately from Theorem \ref{emcg}. \end{proof}

\begin{tm} \label{emdc}
  $E^m$ is dense in $(E^m_{nc}, H_{\rm end})$.
\end{tm}

\begin{proof} \  Take $u\in E^m_{nc}$. Then $\{u^{(1/n)}, n=1,2,\ldots\}$ is a sequence in   $E^m$.
By
Lemma \ref{csn},
$H_{\rm end} (u^{(1/n)}, u) \to 0$ as $n\to \infty$.
So
 $E^m$ is dense in  $(E^m_{nc}, H_{\rm end})$.
\end{proof}

\begin{tm} \label{emcple}
  $(E^m_{nc}, H_{\rm end})$ is the completion of $(E^m, H_{\rm end})$.
\end{tm}

\begin{proof} \   The desired result follows immediately from Theorems \ref{smnc}, \ref{emcg} and \ref{emdc}. \end{proof}

We
summarize
the conclusions on the relationships among various subspaces of the complete space $(F_{USCG}(\mathbb{R}^m), H_{\rm end})$
in Fig. \ref{uscbren}.
\begin{figure}
  \centering
  \includegraphics{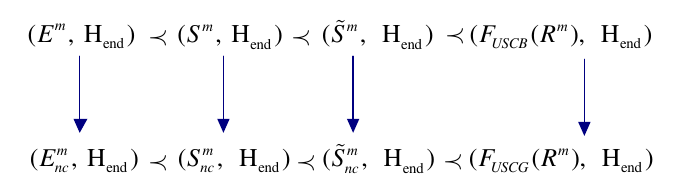}\\
  \caption{The relationships among various subspaces of $(F_{USCG}(\mathbb{R}^m), H_{\rm end})$, where
 $A\prec B$ denotes that $A$ is a closed subspace of $B$ and
$A\longrightarrow B$ means that $B$ is the completion of $A$. }\label{uscbren}
\end{figure}

\subsection{Characterizations of compactness, relative compactness and total boundedness in subspaces
of
$( F_{USCGCON} (\mathbb{R}^m),   H_{\rm end} )$ } \label{csun}

In
this subsection,
based on the conclusions in
Sections \ref{uscghp}, \ref{uscbpy} and Subsection \ref{resucon},
we
give
characterizations of
relatively compact sets,
totally bounded sets,
and
compact sets in subspaces
of
$( F_{USCGCON} (\mathbb{R}^m),   H_{\rm end} )$.

\begin{tm} \label{urction}
Let $U$ be a subset of $\widetilde{S}^m_{nc}$ ($S^m_{nc}$, $E^m_{nc}$).
Then the following statements are equivalent.
\\
(\romannumeral1) $U$ is a relatively compact set in $( \widetilde{S}^m_{nc},   H_{\rm end} )$ ($(S^m_{nc}, H_{\rm end} )$,  $(E^m_{nc}, H_{\rm end} )$).
\\
(\romannumeral2) $U$ is a totally bounded set in $( \widetilde{S}^m_{nc},   H_{\rm end} )$ ($(S^m_{nc}, H_{\rm end} )$,  $(E^m_{nc}, H_{\rm end} )$).
\\
(\romannumeral3) For
 each $\al \in (0,1]$,
  $\bm{U_\al}:=\{ [u]_\al:  u\in U\}$ is a bounded set in $(   K(\mathbb{R}^m),   H     )$.
 \\
 (\romannumeral4) For
 each $\al \in (0,1]$,
  $U_\al$ is a bounded set in $(   K_S(\mathbb{R}^m),   H     )$ ($(   K_S(\mathbb{R}^m),   H     )$,   $(   K_C(\mathbb{R}^m),   H     )$).
\\
(\romannumeral5)
$U(\al)$
is
a bounded set in $\mathbb{R}^m$ for each $\al \in (0,1]$.
\end{tm}

\begin{proof} \ We only prove the desired results when $U$ be a subset of $\widetilde{S}^m_{nc}$. The  the rest of situations can be proved in a similar
manner.

Since
$( \widetilde{S}^m_{nc},   H_{\rm end} )$
is a complete
space,
we know
that
$U$ is a relatively compact set
if and only if
$U$ is a totally bounded set in $( \widetilde{S}^m_{nc},   H_{\rm end} )$.
Thus,
by
Theorem
\ref{fusbtoundcn},
statement (\romannumeral1) and statement (\romannumeral2) are
 both
 equivalent
   to
statement (\romannumeral5).

Note the $[u]_\al \in  K_S(\mathbb{R}^m)$ for each $u\in \widetilde{S}^m_{nc}$ and $\al\in (0,1]$.
Hence
 statement (\romannumeral3) and statement (\romannumeral4) are equivalent. Clearly, these two statements
 are
 also
 equivalent
   to
statement (\romannumeral5).
\end{proof}

If $U \subset \widetilde{S}^m_{nc}$ ($S^m_{nc}, E^m_{nc}$),
then
$\overline{U}$ is exactly the closure of $U$ in $( \widetilde{S}^m_{nc},   H_{\rm end} )$ ($(S^m_{nc}, H_{\rm end} )$,  $(E^m_{nc}, H_{\rm end} )$).

\begin{tm}
$U$ is a compact set
in
$( \widetilde{S}^m_{nc},   H_{\rm end} )$ ($(S^m_{nc}, H_{\rm end} )$,  $(E^m_{nc}, H_{\rm end} )$)
  if and only if
  \\
  (\romannumeral1) \ For each $\al \in (0,1]$,
  $U_\al$ is a bounded set in $(   K_S(\mathbb{R}^m),   H     )$ ($(   K_S(\mathbb{R}^m),   H     )$,   $(   K_C(\mathbb{R}^m),   H     )$), and
  \\
  (\romannumeral2) \ $U =  \overline{U}$.
\end{tm}

\begin{proof} \ The desired result follows immediately from Theorem \ref{urction}. \end{proof}

\begin{tm} \label{umtbction}
Let $U$ be a subset of $\widetilde{S}^m$ ($S^m$, $E^m$).
Then the following statements are equivalent.
\\
(\romannumeral1) \    $U$ is a totally bounded set in $( \widetilde{S}^m,   H_{\rm end} )$ ($( S^m,   H_{\rm end} )$, $( E^m,   H_{\rm end} )$).
\\
(\romannumeral2) \  For each $\al \in (0,1]$,
  $U_\al$ is a bounded set in $(   K_S(\mathbb{R}^m),   H     )$ ($(   K_S(\mathbb{R}^m),   H     )$, $(   K_C(\mathbb{R}^m),   H     )$).

\end{tm}

\begin{proof} \ The desired results follow from the fact that
$\widetilde{S}^m \subset \widetilde{S}^m_{nc}$
($S^m \subset S^m_{nc}$,
$E^m \subset E^m_{nc}$)
and
Theorem \ref{urction}.
\end{proof}

\begin{tm} \label{umrection}
Let $U$ be a subset of $\widetilde{S}^m$ ($S^m$, $E^m$).
Then
  $U$ is a relatively compact set in $( \widetilde{S}^m,   H_{\rm end} )$ ($( S^m,   H_{\rm end} )$, $( E^m,   H_{\rm end} )$)
  if and only if
  \\
  (\romannumeral1) \ For each $\al \in (0,1]$,
  $U_\al$ is a bounded set in $(   K_S(\mathbb{R}^m),   H     )$  ($(   K_S(\mathbb{R}^m),   H     )$, $(   K_C(\mathbb{R}^m),   H     )$), and
\\
 (\romannumeral2) \  $\overline{U} \subseteq  \widetilde{S}^m$ ($S^m$, $E^m$).

\end{tm}

\begin{proof} \ The desired conclusions follow immediately from Theorem \ref{urction}. \end{proof}

\begin{tm}
Let $U$ be a subset of $\widetilde{S}^m$ ($S^m$, $E^m$).
Then
  $U$ is a compact set in $( \widetilde{S}^m,   H_{\rm end} )$ ($(S^m,   H_{\rm end} )$, $(E^m,   H_{\rm end} )$)
  if and only if
  \\
  (\romannumeral1) \ For each $\al \in (0,1]$,
  $U_\al$ is a bounded set in $(   K_S(\mathbb{R}^m),   H     )$ ($(   K_S(\mathbb{R}^m),   H     )$, $(   K_C(\mathbb{R}^m),   H     )$), and
\\
 (\romannumeral2) \  $\overline{U} = U$.

\end{tm}

\begin{proof} \ From Theorem \ref{urction}, we can obtain the desired results. \end{proof}

\begin{tm} \label{subfre} Let $U$ be a subset of $\widetilde{S}^m$ ($S^m$, $E^m$).
Then
  $U$ is a relatively compact set in
  $( \widetilde{S}^m,   H_{\rm end} )$ ($( S^m,   H_{\rm end} )$, $( E^m,   H_{\rm end} )$)
 if and only if
 \\
(\romannumeral1)
  For each $\al \in (0,1]$,
  $U_\al$ is a bounded set in $(   K_S(\mathbb{R}^m),   H     )$  ($(   K_S(\mathbb{R}^m),   H     )$, $(   K_C(\mathbb{R}^m),   H     )$), and
 \\
(\romannumeral2$'$)
Given $\{u_n: n=1,2,\ldots\} \subset U$, there exists a $r>0$ and a subsequence $\{v_n\}$ of $\{u_n\}$ such
that
$\lim_{n\to \infty}    |v_n|_r=0$.

\end{tm}

\begin{proof} \ Note that $F_{USCB} (\mathbb{R}^m) \cap \widetilde{S}^m_{nc} = \widetilde{S}^m$,
$
F_{USCB} (\mathbb{R}^m) \cap S^m_{nc} = S^m
$
and
$F_{USCB} (\mathbb{R}^m) \cap E^m_{nc} = E^m$.
So,
by
Theorem \ref{fbsrchra},
we obtain the desired results.
\end{proof}

\begin{tm} \label{subfc} Let $U$ be a subset of $\widetilde{S}^m$ ($S^m$, $E^m$).
Then
  $U$ is a compact set in
  $( \widetilde{S}^m,   H_{\rm end} )$ ($( S^m,   H_{\rm end} )$, $( E^m,   H_{\rm end} )$)
 if and only if
 \\
(\romannumeral1)
  For each $\al \in (0,1]$,
  $U_\al$ is a bounded set in $(   K_S(\mathbb{R}^m),   H     )$  ($(   K_S(\mathbb{R}^m),   H     )$, $(   K_C(\mathbb{R}^m),   H     )$),
 \\
(\romannumeral2$'$)
Given $\{u_n: n=1,2,\ldots\} \subset U$, there exists a $r>0$ and a subsequence $\{v_n\}$ of $\{u_n\}$ such
that
$\lim_{n\to \infty}    |v_n|_r=0$, and
\\
(\romannumeral3)
$U$ is a closed set in  $( \widetilde{S}^m,   H_{\rm end} )$ ($( S^m,   H_{\rm end} )$, $( E^m,   H_{\rm end} )$).
\end{tm}

\begin{proof} \ The desired conclusions follow immediately from Theorem \ref{subfre}.   \end{proof}

\section{Conclusions}

 The first theme
of
 this paper is the relationships of the endograph metric and the $\Gamma$-convergence.
In general, the   endograph metric   convergence is stronger than
the
$\Gamma$-convergence
on
 $F_{USC} (\mathbb{R}^m)$.
We
show
 that an endograph metric
  convergent sequence in $F_{USCG} (\mathbb{R}^m)$ is exactly
a
$\Gamma$-convergent sequence in $F_{USCG} (\mathbb{R}^m)$ satisfying the condition that
 the union of $\al$-cuts of all its elements is a bounded set in $\mathbb{R}^m$ when $\al>0$ (Section \ref{rghc}).
 Moreover, we
find
that
the
 endograph metric metrizes the $\Gamma$-convergence on $F_{USCGCON} (\mathbb{R}^m) \backslash \widehat{\emptyset}$,
which includes common fuzzy sets
  such as
fuzzy numbers (compact and noncompact), fuzzy star-shaped
numbers (compact and noncompact), and general fuzzy star-shaped numbers (compact and noncompact).
This means
that
the Hausdorff metric metrizes the Fell topology on the endographs of these fuzzy sets (Subsection \ref{sbp}).
We also discuss
the
relationships among
the $d_p$ metric, the endograph metric and the $\Gamma$-convergence (Subsection \ref{regp}).

The second theme is the level characterizations of the endograph metric convergence and   the $\Gamma$-convergence.
This theme involves the level characterizations of fuzzy set itself, which is discussed in Section \ref{lcfu}.
Based on this,
  we show
that
the endograph metric convergence and   the $\Gamma$-convergence
have
 level
 decomposition properties.
Namely,
the endograph metric convergence on $F_{USCG} ( \mathbb{R}^m )$ can be broken down to the Hausdorff metric convergence on some of the corresponding $\al$-cuts,
   and
   the $\Gamma$-convergence on $F_{USC} ( \mathbb{R}^m )$ can be broken down to the Fell topology convergence (Kuratowski convergence) on some of the corresponding $\al$-cuts (Subsection \ref{lcge}).
We also give the forms of these level characterizations
when
restricted on  $F_{USCGCON} (\mathbb{R}^m) \backslash \widehat{\emptyset}$ and common fuzzy sets, respectively (Subsection \ref{sbp}).

The third theme is
characterizations of relative compactness, total boundedness, and compactness in subspaces of $(F_{USCG} (\mathbb{R}^m), H_{\rm end})$ and the relationships among these spaces.
We
clarify the relationships among
various
subspaces
of
$(F_{USCG} (\mathbb{R}^m), H_{\rm end})$
including several general fuzzy set spaces and all common fuzzy set spaces mentioned here.
One of the conclusions
is
that
the fuzzy set spaces of noncompact type   are exactly the completions of their compact counterparts under the endograph metric.
All conclusions on the relationships are summarized in a figure (Subsection \ref{resucon}).
We
 give
characterizations of relative compactness, total boundedness, and compactness
in all
these   fuzzy set spaces (Sections \ref{uscgpre} and \ref{uscbpy} and Subsection \ref{csun}).
A basic
conclusion
is that
a set in   $(F_{USCG} (\mathbb{R}^m), H_{\rm end})$
is relatively compact
  if and only if it is
totally bounded
if and only if
the union of $\al$-cuts of all its elements is a bounded set in $\mathbb{R}^m$ when $\al>0$.

Each theme is closely related to the other two. The discussions of these three themes are carried out alternately.
We first give
the relationship
of
endograph metric
  convergence and $\Gamma$-convergence on
$F_{USCG} (\mathbb{R}^m)$.
From this result and the level decomposition property of $\Gamma$-convergence,
we
obtain the level decomposition property of endograph metric.
Alternately, by using these
 level decomposition properties, we find the endograph metric and the $\Gamma$-convergence
are
coincide on   $F_{USCGCON} (\mathbb{R}^m) \backslash \widehat{\emptyset}$.
This is a further understanding of the relationships between these two convergence structures.

Using the level decomposition property of   endograph metric,
we
 present characterizations of relative compactness, total
boundedness, and compactness
in $(F_{USCG} (\mathbb{R}^m), H_{\rm end})$.
Conversely,
from these characterizations,
we
 can also deduce the relationship of the endograph metric and the $\Gamma$-convergence on
$F_{USCG} (\mathbb{R}^m)$ (See Remark \ref{trm} for details).

The results in this paper
can be used
to
theoretical research
and
real world applications of fuzzy set
such as
 the approximations of fuzzy sets,  the solutions of fuzzy differential equations, the analysis and design of fuzzy systems, and so on.

\section*{Acknowledgement}

The author would like to thank the anonymous referees for their invaluable suggestions and comments
that improve the presentation of this paper.

\appendix

\section{Proof of Theorem \ref{fusdctn}}

We first introduce
the
following auxiliary function.

Let $u\in F(\mathbb{R}^m)$, $t\in \mathbb{R}^m$ and $r$ be a positive number in $\mathbb{R}$.
Define a function
$S_{u,t,r}(\cdot, \cdot): \mathbb{S}^{m-1} \times [0,1] \to \{ -\infty \} \cup \mathbb{R}$
by
$$
S_{u,t,r}(e, \al)= \left\{
                        \begin{array}{ll}
      -\infty,    &     \hbox{if} \ [u]_\al \cap \overline{ B(t, r)}  =\emptyset,
\\
\sup \{    \langle e, x-t \rangle : x\in [u]_\al \cap \overline{ B(t, r)}   \} ,  &  \hbox{if} \ [u]_\al \cap \overline{ B(t, r)}  \not=\emptyset,
                        \end{array}
                      \right.
  $$
where  $\overline{ B(t, r)}$ denote
the closed ball
$\{x\in \mathbb{R}^m: d(t,x) \leq r \}$.

We can see that $S_{u,t,r}(e, \cdot)$ is a monotone function on $[0,1]$, and
$ \lim_{\beta \to \alpha+}S_{u,t,r}(e, \beta) \leq  S_{u,t,r}(e, \alpha)  \leq   \lim_{\beta \to \alpha-}S_{u,t,r}(e, \beta) $.

We say $\alpha\in (0,1)$ is a discontinuous point of $S_{u,t,r}(e, \cdot)$
if
\\
(\romannumeral1)  $S_{u,t,r}(e, \alpha) \in \mathbb{R}$, and
\\
  (\romannumeral2) $S_{u,t,r}(e, \beta) = -\infty$ for all $\beta>\alpha$
or
$-\infty < \lim_{\beta \to \alpha+}S_{u,t,r}(e, \beta)  <   \lim_{\beta \to \alpha-}S_{u,t,r}(e, \beta) $.
\\
Denote the set of all discontinuous points of $S_{u,t,r}(e, \cdot)$ by $D_{u, t, r, e}$.
Then
from
 the monotonicity of $S_{u,t,r}(e, \cdot)$, $D_{u,t,r,e}$ is at most countable.

\begin{lm} \label{dutrcoun}
  Let $u$ be a fuzzy set on $\mathbb{R}^m$, $t$ be a point in $\mathbb{R}^m$, and $r$ be a positive real number.
  Then
  $D_{u,t,r}:=\bigcup_{e\in \mathbb{S}^{m-1}} D_{u,t,r,e}$
  is at most countable.
\end{lm}

\begin{proof} \ Let $\varpi$ be a countable dense subset of $\mathbb{S}^{m-1}$.
Then
\begin{equation}\label{couren}
 D_{u,t,r} =\bigcup_{e\in \varpi} D_{u,t,r,e}
\end{equation}
In fact,
suppose that $\al\in D_{u,t,r}$.
Then there exists
$e\in \mathbb{S}^{m-1}$ such that
 $\alpha \in D_{u,t,r, e}  $.
Hence
$S(u,t,r)(e, \alpha)>-\infty$.
This is equivalent to
$[u]_\alpha \cap \overline{B(t,r)} \not=\emptyset.$
Therefore
\begin{equation}\label{sar}
S(u,t,r)(f, \alpha)>-\infty \   \hbox{for all} \  f\in \mathbb{S}^{m-1}.
\end{equation}
To show $\al \in \bigcup_{e\in \varpi} D_{u,t,r,e}$,
we divide the proof into two cases.

Case 1. \ $S(u,t,r)(e, \beta)=-\infty$
for all
$\beta > \alpha$.

In this case,
$[u]_\beta \cap \overline{ B(t, r)}  =\emptyset$ for any
$\beta > \alpha$,
and so
$S(u,t,r)(f, \beta)=-\infty$ when $f\in \mathbb{S}^{m-1}$ and $\beta > \alpha$.
Combined with \eqref{sar}, we know
$\al \in D_{u,t,r,f}$ for each $f\in \mathbb{S}^{m-1}$.
Thus
$\al \in \bigcup_{e\in \varpi} D_{u,t,r,e}$.

Case 2. \ $-\infty < \lim_{\beta \to \alpha+}S_{u,t,r}(e, \beta)  <   \lim_{\beta \to \alpha-}S_{u,t,r}(e, \beta) $.

In this case, there is an $\al_0 > \al $ such that $[u]_{\lambda} \cap  \overline{B(t,r)} \not= \emptyset $
 when $\lambda \in [0, \al_0]$.
Set
\begin{equation}\label{den}
  \xi : =  \lim_{\beta \to \alpha-}S_{u,t,r}(e, \beta)  -   \lim_{\beta \to \alpha+}S_{u,t,r}(e, \beta)  >0.
\end{equation}
 Notice that, for all $\beta\in [0,1]$ with $[u]_\beta \cap \overline{B(t,r)} \not= \emptyset$,
 \begin{align*}
    | S_{u,t,r}(e, \beta) &-  S_{u,t,r}(f, \beta)|    \\
&=
  | \sup \{    \langle e, x-t \rangle : x\in [u]_\beta  \cap \overline{B(t,r)}  \} - \sup \{    \langle f, x-t \rangle : x\in [u]_\beta  \cap \overline{B(t,r)}  \} |
   \\
    & \leq  \sup \{  \mid \langle  e-f,   x-t \rangle \mid : x\in [u]_\beta \cap \overline{B(t,r)} \}
    \\
    & \leq  \|  e-f \|  \cdot r,
 \end{align*}
hence, for any $\lambda \in [0, \al_0]$,
$$ \mid  S_{u,t,r}(e, \lambda) -  S_{u,t,r}(f, \lambda)\mid \leq  \|  e-f \|  \cdot r,$$
and so, combined with \eqref{den},
we know
there exists $\delta>0$ such that, for all $f\in  \mathbb{S}^{m-1} \cap B(e,\delta)$,
 $$ \lim_{\beta \to \alpha-}S_{u,t,r}(f, \beta)  -   \lim_{\beta \to \alpha+}S_{u,t,r}(f, \beta) > \xi/2. $$
 This means that $\alpha \in D_{u,t,r,f}$ when $f\in  \mathbb{S}^{m-1} \cap B(e,\delta)$.
Thus there exists $g\in \varpi$
such that
$\al\in D_{u,t,r,g}$,
i.e.,
$\al \in \bigcup_{e\in \varpi} D_{u,t,r,e}$.

Now we obtain \eqref{couren}. Since $\varpi$ is countable and $D_{u,t,r,e}$ is at most countable,
we know
$D_{u,t,r}$ is at most countable.
\end{proof}

\begin{re}{\rm
  In the proof of Lemma \ref{dutrcoun}, in order to show $ D_{u,t,r}$ is at most countable,
we prove
that
 $D_{u,t,r}=\bigcup_{e\in \varpi} D_{u,t,r,e}$.
This kind of trick was used in the proof of
Lemma 4 in \cite{trutschnig} to show a set is at most countable.
It is claimed in Lemma 13 of \cite{trutschnig}
that this trick
can be used to prove that
$P(u)$ is at most countable when $u\in E_{nc}^m$.
}
\end{re}

  \noindent \textbf{Theorem 5.1 } \ For each fuzzy set $u$ on $\mathbb{R}^m$,
  the set
 $D(u)=\{\al\in (0,1):  [u]_\al \nsubseteq  \overline{ \{u> \al\}}  \;\}$ is at most countable.

\vspace{3mm}

\noindent\textbf{The proof of Theorem
\ref{fusdctn}}. \
 If
 $\alpha \in D(u)$, then
 there is a
$y\in \mathbb{R}^m$ such that $y\in [u]_\al$
but
$y \notin \overline{ \{u>\al\}}$.
Thus,
\begin{equation}\label{dgvn}
  d(y, \overline{ \{u>\al\}}) > \varepsilon>0.
\end{equation}
Choose a $q \in \mathbb{Q}^m = \{(z_1, z_2,\ldots, z_m) \in \mathbb{R}^m :   z_i \in \mathbb{Q}, \ i=1,2,\ldots, m \}   $ which
satisfies that
 $\|y-q\|>0$. We assure  that $\al \in D_{u,q, r}$ for some $r\in Q$ with $r\geq \|y-q\|$.

In fact, let $e= (y-q)/\|y-q\|  $. Then
\begin{equation}\label{suqrae}
  S_{u,q,r}(e, \al)\geq \langle e, y-q \rangle = \|y-q\|.
\end{equation}
If $[u]_\beta \cap \overline{B(q,r)} = \emptyset$ for   any $\beta>\alpha$,
then
$S_{u,q,r}(e, \beta)=-\infty$ for all $\beta > \alpha$,
and thus
$\al \in D_{u,q, r,e}$.

If    there exists $\beta>\alpha$ such
that
$[u]_\beta \cap \overline{B(q,r)} \not= \emptyset$.
Pick an arbitrary
$x\in \overline{\{u>\al\}} \cap \overline{B(q,r)} $.
Then, by \eqref{dgvn},
\begin{gather*}
  \|x-q\| \leq r,
 \\
\|x-y\| >\varepsilon.
\end{gather*}
If $x=q$, then $ \langle  e, x-q \rangle =0$.
Suppose that $x \not= q$.
Notice that
$$ \frac{\langle y-q , x-q\rangle}{ \|y-q\| \cdot \|x-q \|}=\cos \alpha
=
 \frac{\|x-q\|^2 +  \|y-q\|^2 -\|x-y\|^2 }{2\|y-q\| \cdot   \|x-q\| },$$
where
$\al$ is the angle between two vectors
$x-q$ and $y-q$.
Thus
\begin{align*}
  \langle  e, x-q \rangle & =\frac{1}{2}
\left(     \frac{\|x-q\|^2}{ \|y-q\|} +   \|y-q\|  -  \frac{ \|x-y\|^2}{ \|y-q\|}   \right )
\\
   & \leq   \frac{1}{2} \left(    \frac{r^2}{ \|y-q\|}   +   \|y-q\|  -  \frac{ \varepsilon^2}{ \|y-q\|}   \right ),
\end{align*}
and
so
 there exists a $\delta>0$ such that for all $r \in [ \|y-q\|,  \|y-q\|+\delta  )$,
\begin{equation}\label{suqrbe}
   \langle  e, x-q \rangle   \leq    \|y-q\| -   \frac{1}{4} \frac{ \varepsilon^2}{ \|y-q\|}.
\end{equation}
Combined with \eqref{suqrae} and \eqref{suqrbe},
it then follows from the arbitrariness
of
$x\in \overline{\{u>\al\}} \cap \overline{B(q,r)} $
that
$$\lim_{\beta \to \al+} S_{u,q,r}(e,\beta) <  S_{u,q,r}(e,\alpha)$$
when
$r \in [ \|y-q\|,  \|y-q\|+\delta  )$.
This implies
that
there exists $r\in \mathbb{Q}$
such that
$\al \in D_{u,q,r,e} \subset  D_{u,q,r}$.

Now we know
$$D(u) \subseteq   \bigcup_{q\in \mathbb{Q}^m, r\in \mathbb{Q}} D_{u,q,r}.$$
It then follows from Lemma \ref{dutrcoun}
that
 $D(u)$
is at most countable.

\hfill $\square$

\end{document}